\documentclass[12pt]{article}
\usepackage[utf8]{inputenc}
\usepackage{amsthm}
\usepackage{amsmath}
\usepackage{amssymb}
\usepackage{amscd}
\usepackage[all]{xy}
\usepackage[margin=1in]{geometry}
\usepackage{comment}
\usepackage{inputenc}
\usepackage{pdfpages}
\usepackage{tikz}
\usepackage{pgfplots}
\usepackage{latexsym}
\usepackage{tikz-cd}
\usetikzlibrary{patterns}

\begin{document}

\title{Non-integrally closed Kronecker function rings and integral domains with a unique minimal overring}
\author{Lorenzo Guerrieri \thanks{ Jagiellonian University, Instytut Matematyki, 30-348 Krak\'{o}w \textbf{Email address:} lorenzo.guerrieri@uj.edu.pl }  \and K. Alan Loper
\thanks{ Department of Mathematics, Ohio State University – Newark, Newark, OH 43055  \textbf{Email address:}  lopera@math.ohio-state.edu }}
\maketitle

\begin{abstract}
\noindent It is well-known that an integrally closed domain $D$ can be express as the intersection of its valuation overrings but, if $D$ is not a Pr\"{u}fer domain, the most of valuation overrings of $D$ cannot be seen as localizations of $D$. The Kronecker function ring of $D$ is a classical construction of a Pr\"{u}fer domain which is an overring of $D[t]$, and its localizations at prime ideals are of the form $V(t)$ where $V$ runs through the valuation overrings of $D$.
This fact can be generalized to arbitrary integral domains by expressing them as intersections of overrings which admit a unique minimal overring. In this article we first continue the study of rings admitting a unique minimal overring extending known results obtained in the 70's and constructing examples where the integral closure is very far from being a valuation domain.
Then we extend the definition of Kronecker function ring to the non-integrally closed setting by studying intersections of Nagata rings of the form $A(t)$ for $A$ an integral domain admitting a unique minimal overring.
\medskip

\end{abstract}

\newtheorem{theorem}{Theorem}[section]
\newtheorem{lemma}[theorem]{Lemma}
\newtheorem{prop}[theorem]{Proposition}
\newtheorem{corollary}[theorem]{Corollary}
\newtheorem{problem}[theorem]{Problem}
\newtheorem{construction}[theorem]{Construction}

\theoremstyle{definition}
\newtheorem{defi}[theorem]{Definitions}
\newtheorem{definition}[theorem]{Definition}
\newtheorem{remark}[theorem]{Remark}
\newtheorem{example}[theorem]{Example}
\newtheorem{question}[theorem]{Question}
\newtheorem{comments}[theorem]{Comments}

\newtheorem{discussion}[theorem]{Discussion}

\newcommand{\N}{\mathbb{N}}
\newcommand{\m}{\mathfrak{m}}
\newcommand{\p}{\mathfrak{p}}
\newcommand{\q}{\mathfrak{q}}
\newcommand{\Z}{\mathbb{Z}}
\newcommand{\Q}{\mathbb{Q}}
\newcommand{\al}{\boldsymbol{\alpha}}
\newcommand{\be}{\boldsymbol{\beta}}
\newcommand{\de}{\boldsymbol{\delta}}
\newcommand{\e}{\textbf{e}}
\newcommand{\om}{\boldsymbol{\omega}}
\newcommand{\g}{\boldsymbol{\gamma}}
\newcommand{\te}{\boldsymbol{\theta}}
\newcommand{\he}{\boldsymbol{\eta}}
\def\min{\mbox{\rm min}}
\def\max{\mbox{\rm max}}
\def\ff{\frak}
\def\Spec{\mbox{\rm Spec }}
\def\Zar{\mbox{\rm Zar }}
\def\Proj{\mbox{\rm Proj }}
\def\hgt{\mbox{\rm ht }}
\def\type{\mbox{ type}}
\def\Hom{\mbox{ Hom}}
\def\rank{\mbox{ rank}}
\def\Ext{\mbox{ Ext}}
\def\Ker{\mbox{ Ker}}
\def\Max{\mbox{\rm Max}}
\def\End{\mbox{\rm End}}
\def\xpd{\mbox{\rm xpd}}
\def\Ass{\mbox{\rm Ass}}
\def\emdim{\mbox{\rm emdim}}
\def\epd{\mbox{\rm epd}}
\def\repd{\mbox{\rm rpd}}
\def\ord{\mbox{\rm ord}}
\def\gcd{\mbox{\rm gcd}}
\def\Tr{\mbox{\rm Tr}}
\def\Res{\mbox{\rm Res}}
\def\Ap{\mbox{\rm Ap}}
\def\sdefect{\mbox{\rm sdefect}}
\maketitle

\section{Introduction}

Let $D$ be a local, integrally closed integral domain with maximal ideal $\m$ and quotient field $K$.
Let $\alpha \in K$ be such that $\alpha$ and $1/\alpha$ are not in $D$.  A theorem of Seidenberg \cite[Theorem 7]{seidemberg} indicates
that $\m$ extends to a nonmaximal prime ideal in the ring $D[\alpha]$ and gives details concerning the structure of the
maximal ideals of $D[\alpha]$ which contain $\m$.  Here, we only note that there are infinitely many.
Clearly, these same results hold for $D[1/\alpha]$.
The fact that $D$ is integrally closed insures that $D[\alpha] \cap D[1/\alpha] = D$.  We can think of the introduction
of $\alpha$ and $1/\alpha$ as being like striking a large crystal (the ring $D$) with a hammer and shattering
it into many smaller pieces.


Of course, the scenario described above does not quite always work.  In particular, it is impossible to 
choose the element $\alpha$ when $D$ is a valuation domain.
In accordance with the depiction above, a classical theorem by Krull states that every integrally closed domain is 
the intersection of all its valuation overrings. For this reason, valuation rings are the indivisible atoms in a canonical
decomposition of an integrally closed domain.    This  makes it seem that valuation rings would be
a powerful tool in studying the structure of rings.  This is often true, but there are several things that can make it difficult.


First, note that it is hard to describe in a simple way all the valuation overrings of a given ring.  For instance, let $F$ be any field.
The ring $D = F[x,y]$ of polynomials in two variables has uncountably many valuation overrings, but only 
few of them are easy to see when looking at the localizations of $D$.   Krull provided an answer to this difficulty.  The Kronecker
function ring $Kr(D)$ of $D$ is an overring of the polynomial ring $D[t]$ with the property that that there is an easy correspondence
between the valuation overrings of $D$ and the localizations of $Kr(D)$ at prime ideals \cite{krull1}, \cite{krull2}, \cite{gilmer}.
A second difficulty is harder to overcome.  If one is inclined to analyze a ring by the means of studying its valuation
overrings, it can be disconcerting to be faced with an integral domain that is not integrally closed.  Of course, one can turn
to tools of an entirely different nature.  The aim of this paper is to generalize the notions of valuation
ring and Kronecker function ring to a non-integrally closed setting.  


When semistar operations were first introduced one of the major selling points was the possibility of using semistar operations
to build a Kronecker function ring for a domain which is not integrally closed \cite{fon-lop2}.  However, in this case
the Kronecker function ring that is produced 
is simply a conventional Kronecker function ring of an integrally closed overring of $D$.  Our real purpose is to 
produce a non-integrally closed Kronecker function ring which mirrors the non-integrally closed character of $D$.
The first step in such a process is to generalize the notion of a valuation domain.  We need two properties to hold for such a
generalization.  
\begin{itemize}
\item We need a generalized valuation ring to actually be a valuation ring if it is integrally closed.
\item We need for a domain which is not integrally closed to be equal to the intersection of the generalized
valuation rings that contain it.
\end{itemize}
Suppose now $D$ to be not integrally closed.
Let $\alpha$ be an element 
of the quotient field $K$ which is not in $D$.  Let $A_\alpha$ be an overring of $D$ which is maximal with respect to the property of not
containing the element $\alpha$. Such a domain exists by application of Zorn's Lemma and it is necessarily local. We call $A_\alpha$ a \it maximal excluding domain\rm. It has been observed that maximal excluding domains are exactly the integral domains admitting a unique minimal overring in the sense of \cite{gil-hei} (see also \cite{gil-huc}). This is our proposed notion for a 
generalization of valuation rings.  

The structure of this paper is as follows.
Section 2 is devoted to review the already known facts concerning Kronecker function rings and maximal excluding domains. We provide in this section all the needed references. 

Section 3 is devoted to constructing examples of maximal excluding domains.  Almost all valuation
rings are maximal excluding (specifically all those with branched maximal ideal).  One might think that this fact, combined with the fact that we are specifically working with
rings that are not integrally closed, would mean that the integral closure of a maximal excluding domain is a
valuation ring, or at least a Pr\"ufer domain.  This is known to not be true \cite{gil-hof}.  Here we show how to 
construct maximal excluding domains using generalized power series rings.  Our construction is interesting since the maximal excluding domains constructed can be farther from being Pr\"ufer than examples that are
already in the literature.  We can indeed construct examples that have infinitely many incomparable prime ideals of the same height. 
We also construct maximal excluding domains of a more ordinary character using pullbacks.


In section 4 we work in earnest on the theory of non-integrally closed Kronecker function rings.  In the classical setting
we begin with an integrally closed domain $D$.  In this case, the key feature of $Kr(D)$ is that if we localize at any prime ideal we get
a ring of the form $V(t)$ where $V$ is a valuation overring of $D$ and $V(t)$ is what is known as the Nagata ring of $V$.
Suppose that $D$ is not integrally closed and is expressed as an intersection of maximal excluding domains.  We can define 
a new generalized version of Kronecker function ring to be the intersection of the Nagata rings of the maximal excluding domains.  We should not expect in this 
setting to always recover all the maximal excluding domains back by localizing this generalized Kronecker function ring.  For example, if the 
integral closure of $D$ is a Pr\"ufer domain then perhaps we should expect any of these generalized Kronecker function rings to be just equal to the Nagata ring $D(t)$.
We explore various settings where this happens. In particular we prove that the only Kronecker function ring of $D$ is $D(t)$ if the integral closure of $D$ is Pr\"ufer, semilocal and has sufficiently large residue fields. Along the way we study general properties about whether the operation of Nagata rings extension commutes with intersection of integral domains.


Finally, in Section 5 we explore various settings where: we start with a domain $D$, which is not integrally closed, we express $D$ as an intersection of maximal excluding domains, and we intersect the rings $A(t)$ where $A$ runs through the maximal excluding rings in our collection. We first demonstrate that a localization of this intersection at a maximal ideal has the form $C(t)$ where $C$ is an overring of $D$ (providing also several example where this ring $C$ is maximal excluding). Then, we demonstrate that the integral closure of this intersection is a classical Kronecker function ring of the integral closure of $D$. We analyze also a simpler construction of the form $R = Kr(\overline{D}) \cap A(t)$ where $\overline{D}$ is the integral closure of a domain $D$ and $D = \overline{D} \cap A$ for $A$ a semilocal overring. We can study the properties of this ring $R$ more in general without requiring $A$ to be maximal excluding. To give examples and applications we involve the maximal excluding rings coming from the constructions in Section 3. Furthermore, along all the paper, we left several open questions for further research.








\section{Preliminaries}

We fix our notation for this article and we recall all the definitions and results that we will need about Nagata rings, Kronecker function rings, and integral domains maximal with respect to excluding a given element of their quotient field.

All the rings we consider will be integral domains, having the same unit element $1$. For an integral domain $D$, we denote its quotient field by $\mathcal{Q}(D)$. If $D$ is local, we denote its unique maximal ideal by $\m_D$. An \it overring \rm of $D$ is an integral domain $A$ such that $D \subseteq A \subseteq \mathcal{Q}(D)$. If $A$ is an overring and $A$ and $D$ are both local, we say that $A$ \it dominates \rm $D$ if $\m_A \supseteq \m_D$. The set of all valuation overrings dominating $D$ is standardly called Zar$(D)$ (the name comes from Zariski's definition). Given a valuation overring $V$ we denote by $v$ the associated valuation and by $G_V$ its value group. The integral closure of an integral domain $D$ in a field $F \supseteq \mathcal{Q}(D)$ is denoted by $\overline{D}^{F}$. If $F = \mathcal{Q}(D)$, we simply use the notation $\overline{D}$.

\subsection{Nagata rings and Kronecker function rings}

The main reference we consider for Nagata rings and Kronecker function rings is Gilmer's book \cite{gilmer}. For an historical introduction we refer to the paper \cite{historical} and to its bibliography. These subjects have been widely studied in the last 60 years. There are many other relevant references including \cite{krull1}, \cite{krull2}, \cite{nagata}, \cite{arnold}, \cite{halter-koch}, \cite{fon-lop1}, \cite{fon-lop2}.

Given an integral domain $D$ and an indeterminate $t$ over the quotient field of $D$, consider a polynomial $f(t) \in D[t]$. The content of $f$ is the ideal $c(f) \subseteq D$ generated by the coefficients of $f$.
The \it Nagata ring \rm of $D$ is defined as
$$ D(t):= \left\lbrace \frac{f}{g} \, : \, f,g \in D[t], \, c(g)=D \right\rbrace.  $$
This ring is the localization of the ring $D[t]$ at the multiplicatively closed set generated by the polynomials whose content is equal to the unit ideal.

We will use several known facts about the Nagata ring. When $\m$ is a maximal ideal of $D$, the extension $\m D(t)$ is a maximal ideal of $D(t)$, and $D(t)_{\m D(t)}= D_{\m}(t)$ (cf. \cite[Proposition 2.1]{kang}).
Therefore, since $D= \bigcap_{\m \subseteq D} D_{\m}$, we get $D(t)= \bigcap_{\m \subseteq D} D_{\m}(t).$
We can thus say that the operation of Nagata ring extension commutes with the intersection of localizations at the maximal ideals.  
By \cite[Theorem 3]{gil-hof2}, Nagata ring extension commutes also with integral closure. We have that the integral closure of $D(t)$ in its quotient field is $\overline{D}(t)$.

If $V$ is a valuation domain, the ring $V(t)$ is also a valuation domain, called the \it trivial extension \rm of $v$ to the field $\mathcal{Q}(V)(t)$. The value of a polynomial $f= \sum_{k=0}^n a_k t^k \in \mathcal{Q}(V)[t]$ with respect to this valuation is equal to $\min_{k=0,\ldots, n} \lbrace v(a_k) \rbrace$.

\medskip

Let $D$ be an integrally closed domain. By a classical theorem of Krull, $D = \bigcap_{V \in \mbox{ \footnotesize Zar}(D)} V$. 
The \it Kronecker function ring \rm of $D$ can be defined as the intersection 
$$  Kr(D)=\bigcap_{V \in \mbox{\footnotesize Zar}(D)} V(t).   $$
In the literature, the definition of Kronecker function ring is commonly given using the e.a.b star operations of the ring $D$, but for the purpose of this paper, where we study intersections of Nagata rings, we give this as an equivalent definition. It is well-known that the Kronecker function ring of $D$ is always a Bezout domain (indipendently of the properties of $D$) and its localizations at the prime ideals are all the trivial extensions of the valuation overrings of $D$. Moreover, $Kr(D) \supseteq D(t)$ and they coincide if and only if $D$ is a Pr\"{u}fer domain. Given a subset $\mathcal{F} \subseteq \Zar(D)$ such that $D = \bigcap_{V \in  \mathcal{F}} V$, one can define another ring $Kr^{\mathcal{F}}(D)=\bigcap_{V \in \mathcal{F}} V(t).$ This last ring is an overring of the Kronecker function ring $Kr(D)$. Also these overrings $Kr^{\mathcal{F}}(D)$ are commonly called Kronecker function rings of $D$.

\subsection{Integral domains maximal with respect to excluding an element of their quotient field}

In this paper we are interested in studying intersections of Nagata rings of overrings of an integral domain. For this reason we need to understand the properties of those integral domains which cannot be expressed as intersection of proper overrings. These rings have been already considered in the literature. Gilmer and Heinzer in \cite{gil-hei} consider integral domains admitting a unique minimal overring, in the sense that an integral domain $D$ has an overring $A$ such that for any other overring $B$ of $D$ there are inclusions $D \subseteq A \subseteq B$. Several properties of these domains have been studied in \cite{gil-hei}, \cite{gil-huc} and other papers, but generally these rings are still quite mysterious and difficult to identify (in the case they are not integrally closed). It is observed that an integral domain $D$ admits a unique minimal overring in the sense of Gilmer and Heinzer if and only if it is maximal with respect to the property of not containing some element $\alpha \in \mathcal{Q}(D)$ (obviously if $D$ is maximal with respect to excluding $\alpha$, then $D[\alpha]$ is the unique minimal overring of $D$). Clearly these properties are also equivalent to the fact that $D$ cannot be written as intersection of proper overrings. In particular any such a ring $D$ has to be local and its maximal ideal a $t$-ideal (see \cite{gilmer} for information about the star operation $t$). It is also easy to observe that any integral domain can be written as the intersection of its overrings that are maximal with respect to excluding some element of the quotient field.

In \cite{papick}, Papick consider a slightly weaker notion of unique minimal overring. To avoid confusion with this notion we call an integral domain maximal with respect to excluding an element of its quotient field, a \it maximal excluding domain\rm. For a survey about this topic and about the more general concept of minimal ring extensions we refer to \cite{picavet}.

If not otherwise specified, the proofs of all the results that we mention in the following paragraphs can be found in \cite{gil-hei}, \cite{gil-huc}.
The first important known fact is that an integrally closed maximal excluding domain is necessarily a valuation domain with branched maximal ideal (in particular any valuation domain of finite dimension is maximal excluding). Another relevant class of maximal excluding domains, not necessarily integrally closed, is the class of local domains such that every ideal is divisorial, for instance local Gorenstein noetherian domains of dimension one (cf. \cite{heinzer}, \cite{FGH}, \cite{warfield}).

If a domain $D$, maximal with respect to excluding an element $\alpha$, is not integrally closed, then the unique minimal overring $D[\alpha]$ is an integral extension of $D$ and is semilocal with at most two maximal ideals. Furthermore, $\m_DD[\alpha] \subseteq D$.

If $D[\alpha]$ has exactly two maximal ideals, then the integral closure of $D$ is a Pr\"{u}fer domain obtained as the intersection of two valuation rings. In this case, if $D[\alpha]= \overline{D}$, the structure of $D$ has been described in \cite[Theorem 14]{jaballah} with the use of pullback diagrams.
Also if the maximal ideal $\m_D$ of $D$ coincides with the maximal ideal of $D[\alpha]$, then the integral closure of $D$ is a valuation overring $V$ of $D$ such that $\m_V = \m_D$. In the case $D[\alpha]$ is local and its maximal ideal is strictly larger than $\m_D$, the integral closure of $D$ may not be a Pr\"{u}fer domain. The most known example of this situation is the subject of the paper \cite{gil-hof}. In this case $D$ is one-dimensional and its integral closure is a PVD but not a valuation domain (a PVD is a local domain sharing its maximal ideal with a valuation overring \cite{pvd}). A similar example appears here as Example \ref{ex3}. 

Many questions about non-integrally closed maximal excluding domains are still open. Already in \cite{jaballah}, one can find examples of maximal excluding domains whose prime ideals are not linearly ordered (but the integral closure is a Pr\"{u}fer domain). In Section 3 of this paper we show that a maximal excluding domain of dimension at least 2 can have infinitely many prime ideals of the same height and 
its integral closure may not even be a PVD 
(see Theorem \ref{wildintegralclosure}, Example \ref{ex4}).

The \it complete integral closure \rm of a domain $D$ is the ring of the elements $x \in \mathcal{Q}(D)$ for which there exists $d \in D$ such that $dx^n \in D$ for every $n \geq 1$. The complete integral closure of the example described by Gilmer and Hoffmann in \cite{gil-hof} is a valuation domain. The same happens for the classes of maximal excluding domains that we consider here in Section 3. 
We leave the following open question.

\begin{question}
Let $D$ be a maximal excluding domain. Is the complete integral closure of $D$ always a Pr\"{u}fer domain? If $D$ is one-dimensional and $\overline{D}$ is local, is $\overline{D}$ a PVD?
\end{question}



\section{Constructions of maximal excluding domains}

In this section we characterize maximal excluding domains that can be obtained as pullbacks and as generalized power series rings over a field, defined by submonoids of the positive part of a totally ordered abelian group. In both cases, as a consequence of Theorems \ref{pullbacksquare} and \ref{cic}, we get that the complete integral closures of the rings in these families are valuation rings. 
 
\subsection{Maximal excluding domains obtained as pullbacks}

Let $T$ be a local domain with maximal ideal $\m$ and let $B$ be an integral domain having quotient field $ \kappa:= \frac{T}{\m}.$ Let $\phi: T \to \kappa$ be the canonical surjective map. Define the ring $D:= \phi^{-1}(B)$ as in the pullback diagram:
\begin{center}
\begin{tikzcd}
   D \arrow[rightarrow]{r}\arrow[hookrightarrow]{d} 
  &  B \arrow[hookrightarrow]{d}
  \\ 
   T \arrow{r}
 &  \kappa
\end{tikzcd}
\end{center}

For exhaustive information about the properties of rings of this form the reader may consult \cite{gilmer}, \cite{Gabelli-Houston}.

\begin{theorem}
\label{pullbacksquare}
The integral domain $D$ is maximal excluding if and only if $B$ is maximal excluding and $T$ is a valution domain.
\end{theorem}

\begin{proof}
First suppose that $B$ is not maximal excluding. Thus, there are two proper overrings $B_1, B_2$ of $B$ such that $B= B_1 \cap B_2$. It follows that $D = \phi^{-1}(B_1) \cap \phi^{-1}(B_2)$ is an intersection of two proper overrings, hence is not maximal excluding. 

Thus, suppose that $B$ is maximal excluding with unique minimal overring $B[\alpha']$. Let $\alpha$ be an element of $T$
such that $\phi(\alpha) = \alpha'.$ Clearly $\alpha$ is a unit in $T$.

Assume that $T$ is a valuation ring and pick $z \in \mathcal{Q}(D) \setminus D$. In particular $z \not \in \m$. We need to show that $\alpha \in D[z]$. 
If $z \not \in T$ then, $\frac{1}{z} \in \m \subseteq D$. It follows that $ \frac{\alpha}{z} \in \m \subseteq D $ and 
$\alpha = z \frac{\alpha}{z} \in D[z]$.
Next assume $z$ to be a unit in $T$. Therefore $\phi(z) \in \kappa$. Since $B$ is maximal excluding, we get $\alpha' \in B[\phi(z)]$. Since $\m \subseteq D$, we get $ \alpha \in \phi^{-1}(B[\phi(z)]) = D[z].  $ 

Finally consider the case where $T$ is not a valuation domain. Let $V$ be a valuation overring of $T$ such that $\m \subseteq \m_V$. By way of contradiction suppose that $D$ is maximal with respect to excluding an element $\beta$. We must have $\beta \in D[\alpha] \subseteq T$.
If there exists $z \in \m_V \setminus \m$, then we find the contradiction $\beta \not \in D[z]$. Indeed, if $\beta \in D[z]$, then there exists $d_0 \in D$ such that $\beta - d_0 = z(d_1 + zd_2 + \ldots + z^{n-1}d_n) \in \m_V \cap T = \m \subseteq D$. This is a contradiction since $\beta \not \in D$.
If instead $\m_V = \m$, let $\phi'$ be the canonical quotient map $V \to V/\m_V$. Observe that $\phi'|_T = \phi$ and the image of $\phi'$ is a proper field extension of $\kappa$.
We show that there exists an element $z \in V \setminus T$ such that $\kappa(\phi'(z)) \supsetneq \kappa$ and $\beta \not \in D[z]$. Setting $\theta:= \phi'(z)$ for some $z \in V \setminus T$, this is equivalent to showing that $\alpha' \not \in B[\theta]$. If $\theta$ is transcendental over $\kappa$, this is obvious. If it is algebraic, let us assume it to be of degree $n$. Let $f$ be the minimal polynomial of $\theta$ over $F$. Using that $\kappa$ is the quotient field of $B$, we can find $b \in B$ such that the constant term of the minimal polynomial of $b \theta$, which is equal to $b^nf(0)$, is in $B$.
Replace $\theta$ by $b\theta$ to have $f(0) \in B$. Assume there exists a relation $\alpha' = b_0 + b_1 \theta + b_2 \theta^2 + \ldots + b_{t} \theta^{t}$. For $k \geq n,$ using the relation given by $f(\theta)=0$, the term $\theta^{k}$ can be replaced by a linear combination of $\theta, \ldots, \theta^{n-1}$ with coefficient in $\kappa$ and constant term in $B$. This yields another equation of algebraic dependence of $\theta$ over $\kappa$ of degree at most $n-1$ but with constant term not in $B$. This is a contradiction.
\end{proof}

\begin{corollary}
\label{k+mcase}
With the notation of the previous theorem assume that $B$ is a field and $\kappa$ a field extension. Then $D$ is maximal excluding if and only if $T$ is a valuation domain and the extension $\kappa/B$ is a minimal field extension. 
\end{corollary}

\begin{proof}
The condition that $\kappa/B$ is a minimal field extension is clearly necessary. The other conditions can be proved exactly as in Theorem \ref{pullbacksquare} choosing $\alpha'$ to be any element of $\kappa \setminus B$.
\end{proof}



\subsection{Maximal excluding domains in generalized power series rings}

In the following $G$ will denote a totally ordered abelian group (with additive notation and order relation $\leq$), $\boldsymbol{0}$ will denote its unit element and $G_{\geq 0} $ the subsemigroup of elements of $G$ larger than or equal to $\boldsymbol{0}$. We consider a family of generalized power series rings, defined according to \cite{genpowerseries}. The literature about generalized power series rings is very extensive, an interested reader can consult also \cite{ribemboim}, \cite{kim1}, \cite{kim2} and several other papers by P. Ribemboim and other authors.

Let $S$ be a subsemigroup of $G_{\geq 0} $ containing $\boldsymbol{0}$. A subset $A$ of $S$ is Artinian if does not contain any infinite descending sequence of elements with respect to $\leq$ (equivalently, if $A$ admits a minimal element).
Let $K$ be any field. The generalized power series ring $[[K^S]]$ is defined formally as the set of all the maps $S \to K$ such that the set $\mbox{supp}(f):= \lbrace s \in S \, | \, f(s) \neq 0  \rbrace$ is Artinian. The operations on this ring are pointwise addition and the convolution product defined as in \cite[Section 2]{genpowerseries}. The ring $[[K^S]]$ is a commutative integral domain with unit element equal to the map $e$ such that $e(\boldsymbol{0})= 1$ and $e(s)= 0$ for $s \neq \boldsymbol{0}$.
Notice that in \cite{genpowerseries} and in the subsequent papers, the authors consider a more general case of this construction where the group $G$ is not necessarily totally ordered and $K$ is replaced by an arbitrary commutative ring $R$.
The classical semigroup ring $K[X^s, s \in S]$ can be embedded canonically in $[[K^S]]$ sending the homogeneous element $X^s$ to the map $f$ such that $f(s)=1$ and $f(g) = 0$ for every $g \neq s$. For simplicity of notation we also denote the image of this element in $[[K^S]]$ by $X^s$ and we say that these elements are \it monomial elements\rm. A general element of $[[K^S]]$ can be now expressed as a possibly infinite sum $f= \sum_{s \in S} u_s X^s $ with $u_s \in K$ and such that $\mbox{supp}(f)$ is Artinian. 
The units of $[[K^S]]$ can be described using \cite[Proposition 5]{genpowerseries} (which is based on a result proved in \cite{hahn}). It turns out that $f$ is a unit if and only if $f(\boldsymbol{0})$ is a nonzero element of $K$. This fact makes $[[K^S]]$ a local domain. It is easy to observe that, if $S = G_{\geq 0}$, the ring $[[K^{G_{\geq 0}}]]$ is a valuation domain with value group $G$ and the value of an element $f \in [[K^{G_{\geq 0}}]]$ is exactly the minimal element of $\mbox{supp}(f)$. Let us set $V:= [[K^{G_{\geq 0}}]]$.

We want to describe the maximal excluding domains of the form $[[K^S]]$ with $S \subseteq G_{\geq 0}$. In the following we always assume that $[[K^S]]$ and $V$ have the same quotient field.
The first result we prove is inspired by the case $G= \mathbb{Z}$ which has been considered in \cite{FGH}. If $G= \mathbb{Z}$, $V$ is isomorphic to a standard power series ring in one variable over a the field $K$. The set $S$ is a numerical semigroup and the subring $[[K^S]]$ is maximal excluding if and only if $S$ is symmetric or pseudo-symmetric.

\begin{theorem}
\label{maxexclseries} 
Let $S$ be a proper submonoid of $G_{\geq 0} $ and let $D:= [[K^S]] \subsetneq V$. Assume that $D$ and $V$ have the same quotient field. Take $a \in G_{\geq 0} \setminus S $. The following conditions are equivalent:
\begin{enumerate}
\item[(1)] $D$ is maximal with respect to excluding the element $X^a$.
\item[(2)] For every $g \in G$, $g \neq \frac{a}{2}$, $X^g \in D$ if and only if $X^{a-g} \not \in D$.
\end{enumerate}
\end{theorem}

\begin{proof}
Let $v$ be the valuation associated to $V$ and call $\m_V$ the maximal ideal of $V$. Assume $D$ to be maximal with respect to excluding the element $y= X^a$.  
Let us first show that $y\m_V \subseteq D$.
If there exists $z \in y\m_V \setminus D$, then by definition of maximal excluding domain $y = d_0 + d_1z + \ldots + d_n z^n $ with $d_0, \ldots, d_n \in D$. Hence $v(y-d_0) \geq v(z)> v(y)$, implying that $d_0$ is an element of $D$ of value $v(y)=a$. This is a contradiction. It follows that if $g > a$, then $X^{g} \in D$ and if $g < \boldsymbol{0}$, then $X^{a-g} \in D$.

Consider now $g \in G$ such that $\boldsymbol{0} < g < a $ and $2g \neq a$.  If both $X^{g}, X^{a-g} \in D$, then $y =X^{g} X^{a-g} \in D$, contradicting the assumption. 
Suppose both $X^{g}, X^{a-g} \not \in D$. In this case we show that one among $D[X^{g}]$ and $D[X^{a-g}]$ is an overring of $D$ not containing $y$. Suppose $y \in D[X^{g}]$. Hence
$$  X^a = d_0 + d_1 X^{g} + d_2 X^{2g} + \ldots + d_n X^{ng},  $$ with $d_0, \ldots, d_n \in D$. Using that $X^g \in D$ for every $g > a$, if $2g > a$
the equation reduces to $ X^a = d_0 + d_1 X^g$. But this yields a contradiction since no term of $d_1$ can contain $X^{a-g}$ and no term of $d_0 $ can contain $X^{a}$. If instead $2g < a$, we obtain that $y  \not \in D[X^{a-g}]$. For this write $$  X^a = c_0 + c_1 X^{a-g} + c_2 X^{2a-2g} + \ldots + c_n X^{na-ng},  $$ with $c_0, \ldots, c_n \in D$. Now observe that $2a-2g > a$ and conclude in the same way as before.

Assume now condition $(2)$ and prove that $D$ is maximal with respect to excluding $y$. 
Using that $D$ is a subring of $V$, we know that $D$ cannot contain $X^g$ with $g < 0$. Hence $X^g \in D $ for every $g > a$. In particular $D$ contains all the elements of $V$ of value larger than $a$, thus $y\m_V \subseteq D$.  

Let $z$ be an element of the quotient field of $D$. We want to show that $y \in D[z]$. 
In the case $z \not \in V$, we get $\frac{y}{z} \in y\m_V \subseteq D$. Thus $y = z \frac{y}{z} \in D[z] $.
Otherwise, pick $z \in V \setminus D$. By subtracting elements of $D$, we reduce to the case where $v(z)= \mbox{min supp}(z)\in G_{\geq 0} \setminus S.$
Set $b:=v(z)$. Then $z = uX^b + h$ with $u \in K$ and $v(h)> b$.
By hypothesis $X^{a-b} \in D$, hence
$ X^{a-b}z = uy+h' \in D[z]    $. But $v(h')> (a-b)+b=a$, implying $h' \in D$. It follows that $y \in D[z]$.
\end{proof}

\begin{example}
\label{ex1}
Let $G$ be a submonoid of $(\mathbb{R}, +)$. Let $a$ be a positive real number. Set $$ S= \lbrace \boldsymbol{0} \rbrace \cup \left\lbrace g \in G \, | \, \frac{a}{2} < g < a  \right\rbrace \cup \lbrace g \in G \, | \,  g > a  \rbrace. $$
Clearly $S$ is a monoid contained in $G_{\geq 0} $ and by Theorem \ref{maxexclseries}, the ring $D:= [[K^S]]$ is maximal excluding with unique minimal overring $D[X^a]$. It is easy to check that the integral closure of $D$ is $V:= [[K^{G_{\geq 0}}]]$.
\end{example}

Before presenting more examples we describe how to compute the integral closure of certain domains of the form $[[K^S]]$. This result is similar to the corresponding case for semigroup rings, where the integral closure coincides with the root closure, see \cite{rootclosure}, \cite{gradedintegral}.
For this, we need to introduce some new notation and from now on we assume $\dim(V)= n < \infty$.
Write the group $G \cong G_1 \oplus \ldots \oplus G_n$. For $i=1, \ldots, n$ define the semigroup
$$ \widehat{G_i}= \lbrace (g_1, \ldots, g_n) \in G \, | \, g_k = 0 \mbox{ for } k < i \mbox{ and } g_i > 0 \rbrace \subseteq G_{\geq 0}. $$
Set also $\widehat{G_0} = \lbrace \boldsymbol{0} \rbrace$. Observe that $G_{\geq 0}= \bigcup_{i=0}^n \widehat{G_i}$ and $ \widehat{G_i} + \widehat{G_j} \subseteq \widehat{G}_{\footnotesize\mbox{min}(i,j)} $ for any $i,j$. Denote by $\q_i$ the prime ideal of $V$ of height $i$. Then $g \in \widehat{G_i}$ if and only if $X^g \in \q_i \setminus \q_{i-1}.$
Given any submonoid $S$ of $ G_{\geq 0} $, we define $S_i := S \cap \widehat{G_i}$. Clearly $S= \bigcup_{i=0}^n S_i$.
Define also $$ \overline{S}:= \bigcup_{n \in \mathbb{N}} \lbrace g \in G \, | \, ng \in S  \rbrace.$$ 
Also $\overline{S}$ is a monoid contained in $G_{\geq 0}$.
It is straightforward to see that
for every $i=1, \ldots, n$, $\overline{S_i}= \overline{S} \cap \widehat{G_i}$.


\begin{prop}
\label{rootclosure2}
Let $S$ be a proper submonoid of $G_{\geq 0} $ such that $G= \langle S \rangle$. Let $D:= [[K^S]].$ Suppose that for every $i=1, \ldots, n$, there exists $a_i \in S_i$ such that $ \lbrace g \in \overline{S_i} \, | \, g \geq a_i  \rbrace \subseteq S_i. $ Suppose also that the ring $R = [[K^{\overline{S}}]]$ is integrally closed.
Then $R$ is the integral closure of $D$. 
\end{prop}

\begin{proof}
For $f \in R$, we show that $f$ is integral over $D$. We can write $f= f_1+ \ldots + f_n$ in such a way that the support of $f_i$ is in $\overline{S_i} \cup \lbrace  \boldsymbol{0} \rbrace$. It is sufficient to show that any $f_i$ is integral over $D$. But for this, if $g_i$ is the minimum of the support of $f_i$, by the Archimedean property of real numbers one can find a positive integer $e_i$ such that $e_ig_i \geq a_i$. In this way by assumption we get $f_i^{e_i} \in D$.
\end{proof}

We can immediately describe the integral closure in the case when $V$ is one-dimensional.

\begin{corollary}
\label{onedim}
Let $S$ be a submonoid of $G_{\geq 0} $ and let $D:= [[K^S]] \subseteq V= [[K^{G_{\geq 0}}]]$. Suppose that $\dim(V)=1$, $D$ is maximal excluding and $D$ and $V$ have the same quotient field. Then $V$ is the integral closure of $D$.
\end{corollary}

\begin{proof}
In this case the group $G$ is a submonoid of $(\mathbb{R}, +)$.
 By Theorem \ref{maxexclseries}, there exists $a \in G_{\geq 0}$ such that $g \in S$, for every $g > a$. 
 Given $s \in G_{\geq 0}$, 
 there exists some $n \in \mathbb{N}$ such that $ns > a$. Hence, $\overline{S} = G_{\geq 0}$ and $V$ is the integral closure of $D$ by Proposition \ref{rootclosure2}.
\end{proof}

In the next two examples the valuation ring $V$ has dimension 2.

\begin{example}
\label{ex2}
Let $G= \Z \times \Z$, ordered by lexicographic order, and let $ V= [[K^{G_{\geq 0}}]]. $ The ring $V$ is a discrete valuation ring of rank 2. Call $X$ the monomial element of $V$ of value $(0,1)$ and $Y$ the monomial element of $V$ of value $(1,0)$. Consider the submonoid $$ S= \left\lbrace (0,n) \, | \, n \geq 0   \right\rbrace \cup \left\lbrace (1,n) \, | \, n \geq 1   \right\rbrace \cup \left\lbrace (m,n) \, | \, m \geq 2, n \in \Z   \right\rbrace \subseteq G_{\geq 0}. $$ By Theorem \ref{maxexclseries} the ring $D= [[K^S]]= K[[Y\m_V, X]]$ is maximal excluding with unique minimal overring $D[Y]$. The integral closure is clearly $V$ by Proposition \ref{rootclosure2}, since $2(1,n) \in S$ for every $n \in \Z$.
By an analogous argument also $D[Y]$ is maximal excluding with unique minimal overring $D[\frac{Y}{X}]$. Iterating, we observe that there exists an infinite ascending chain 
$$  D \subseteq D[Y] \subseteq D\left[\frac{Y}{X}\right] \subseteq D\left[\frac{Y}{X^2}\right] \subseteq \ldots \subseteq \bigcup_{n=0}^{\infty}D\left[\frac{Y}{X^n}\right] =  V,  $$ such that each ring is maximal excluding, the next one is the unique minimal overring, and the integral closure of all these ring is $V$.
In this example, the ring $D$ has dimension 2 and two nonzero prime ideals. The maximal ideal $\m$ is generated by $X$ and $YX$, the height one prime is $P= \p \cap D$ where $\p$ is the height one prime of $V$. The quotient $D/P$ is a DVR with maximal ideal generated by the image of $X$. The localization $D_P $ coincides with the localization $V_{\p}$. The quotient ring $D/(X)$ is an Artinian ring having dimension 2 as $K$-vector space. 
\end{example}

\begin{example}
\label{ex3}
This example is a sort of dual of the previous one. Again let
$G= \Z \times \Z$, ordered by lexicographic order, let $ V= [[K^{G_{\geq 0}}]] $ and define $X, Y$ as in Example \ref{ex2}. Set $$ S= \left\lbrace (0,0)   \right\rbrace \cup \left\lbrace (1,n) \, | \, n \in \Z \setminus \lbrace 0 \rbrace   \right\rbrace \cup \left\lbrace (m,n) \, | \, m \geq 2, n \in \Z   \right\rbrace \subseteq G_{\geq 0}. $$ It is easy to check that $S$ is a monoid.
 The ring $D= [[K^S]]= K[[Y\m_V, \frac{Y}{X}, \frac{Y}{X^2}, \ldots]]$ is maximal excluding by Theorem \ref{maxexclseries} and the unique minimal overring is $D[Y]$. We notice that this ring $D[Y]$ is equal to $K + YK((X))[[Y]]$, thus it is an integrally closed PVD but not a valuation domain ($X, X^{-1} \not \in D[Y]$).
 By Proposition \ref{rootclosure2}, we observe that $D[Y]$ is the integral closure of $D$. Indeed, $Y$ is clearly integral over $D$ while all the powers of $X$ are not integral since for every $m, n \in \N$, $m(0, n) = (0, mn) \not \in S$. 
 Let $W$ be the rank one valuation overring of $V$. The maximal ideal of $W$ coincides with the height one prime $\p$ of $V$ and is also equal to the maximal ideal $\m$ of $D[Y]$. 
 Hence, in this example $D$ is one-dimensional and its complete integral closure $W$ is a valuation domain.
\end{example}

\begin{remark}
\label{remarkpullback}
Consider the same notation used above and assume $V$ 
countably many prime ideals.
In the next part of this section we can always reduce to assume that if $D$ is maximal excluding with unique minimal overring $D[Y]$, then $Y$ is an element of the height one prime of $V$. Indeed, given any non-maximal prime ideal $\q$ of $V$, it is well-known that $V$ is the pullback of the valuation ring $V/\q$ with respect to the quotient map from $V_{\q}$ to its residue field. The ring $V/\q$ is maximal excluding, hence if the ideal $YV$ has height $i > 1$, we can choose $\q$ to be the prime ideal of height $i-1$ and apply Theorem \ref{pullbacksquare}, reducing to the case where the height of $YV$ is $1$.
\end{remark}

In all the above examples, the complete integral closure of $D$ is the rank one valuation overring of $V$. We prove now that this happens in general for all the maximal excluding rings in this family, provided that they share the same quotient field with $V$ and $V$ has countably many prime ideals.

\begin{theorem}
\label{cic}
Let $S$ be a submonoid of $G_{\geq 0} $ and let $D:= [[K^S]] \subseteq V= [[K^{G_{\geq 0}}]]$. Suppose that $D$ is maximal excluding, $D$ and $V$ have the same quotient field, and $V$ has countably many prime ideals. Then the rank one valuation overring $W$ of $V$ is the complete integral closure of $D$.
\end{theorem}

\begin{proof}
The ring $W$ is completely integrally closed. Hence, it is sufficient to show that every element of $W$ is almost integral over $D$. If $\dim(V)=1$, then $V=W$ and the result follows by Corollary \ref{onedim}. Assume $\dim(V) \geq 2$ and let $\p$ be the height one prime ideal of $V$, which coincides with the maximal ideal of $W$. We can write $G \cong G_W \oplus G_{V/\p}$ where $G_W$ is the rank one value group associated to $W$ and $G_{V/\p}$ is the value group of the valuation ring $V/\p$. 

Say that the unique minimal overring of $D$ is $D[y]$ with $y \in \m_V$. Let $v(y)= (a_1, a_2)$ with $a_1 \geq 0$.
Choose $f \in \p$ such that $v(f) = (s_1, s_2) $ with $s_1 > a_1$. In particular $v(f) > v(y)$. 
Since by Theorem \ref{maxexclseries}, $y\m_V \subseteq D$ we get $f \in D$. 
Pick $h \in W$. If $h \in V$, then for every $n \geq 1$, $v(h^nf) \geq v(f) > v(y) $. If $h \not \in V$, then $v(h)= (0, -g)$, for $g$ a positive element of $G_{V/\p}$. 
Thus $v(h^nf) = (s_1, -ng+s_2) > (a_1, a_2) = v(y)$. In both cases this implies $h^nf \in D$.
\end{proof}

We now analyze further the integral closure of $[[K^S]]$ to show that the integral closure of a maximal excluding domain can be very far from being Pr\"{u}fer. We show that it does not even need to be a PVD and can have infinitely many incomparable prime ideals. For this, we 
assume $1 < \dim(V)= n < \infty$. 
In the following, we let the semigroups $\widehat{G_i}$ and $S_i$ for $i=1, \ldots, n$ be defined as above in this section.
We show how given arbitrary $S_2, \ldots, S_n$ we can construct a monoid $S$ such that $D= [[K^S]]$ is maximal excluding.
By Remark \ref{remarkpullback} we can restrict to consider only the case where the excluded element is in the height one prime of $V$.

\begin{lemma}
\label{importantlemma}
Let $G\cong G_1 \oplus \ldots \oplus G_n$ be a totally ordered abelian group of rank $n > 1$. For each $i=2, \ldots, n$, fix a (possibly empty) semigroup $S_i \subseteq \widehat{G_i}$. Suppose that $ S_i + S_j \subseteq S_{\footnotesize\min(i,j)} $ for any $i, j = 2, \ldots, n$.
Also set $S_0:= \widehat{G_0} = \lbrace \boldsymbol{0} \rbrace$.
Then, it is possible to construct $S_1 \subseteq \widehat{G_1}$ such that $S := \bigcup_{i=0}^n S_i$ is a monoid and the ring $D= [[K^S]]$ is maximal excluding. Moreover, $D$ has the same quotient field of $V= [[K^{G_{\geq 0}}]]$.
\end{lemma}

\begin{proof}
Fix $a= (a_1, \ldots, a_n) \in \widehat{G_1}$ (in particular $a_1 > 0$).
Define $S_1 = S_1' \cup S_1'' \cup S_1^*$ where $$ S_1' = \lbrace g \in \widehat{G_1} \, | \, g > a \rbrace, \quad S_1'' = \lbrace g \in \widehat{G_1} \, | \, g < a, \, g_1= a_1, \mbox{ and } a-g \not \in S_2 \cup \ldots \cup S_n \rbrace, $$ and $S_1^* \subseteq \lbrace g \in \widehat{G_1} \, | \, g_1 < a_1 \rbrace$ is empty if $\lbrace g \in \widehat{G_1} \, | \, g_1 < a_1 \rbrace = \emptyset$ or otherwise is constructed in such a way that:
\begin{itemize}
\item $g \in S_1^*$ if and only if $a-g \not \in S_1^*$.
\item $S_1^* + S_1^* \subseteq S_1$.
\item $S_j + S_1^* \subseteq S_1^*$ for every $j=2, \ldots, n$.
\end{itemize}  
To get some ideas on how one can concretely construct such $S_1^*$ see Remark \ref{remS1}.
We prove now that $S$ is a monoid. 

We have to show that for every $j=1, \ldots, n$, the set $S_j + S_1 \subseteq S_1$. Clearly $S_j + S_1' \subseteq S_1'$ for every $j$. Also $(S_1''+ S_1'') \cup (S_1''+ S_1^*) \subseteq S_1'$. By the hypothesis on $S_1^*$, to conclude we only need to check the inclusion $S_j + S_1'' \subseteq S_1$ for $j \geq 2$. 
Pick $g = (0, \ldots, 0, g_j, \ldots, g_n) \in S_j$ and $h=(a_1, h_2, \ldots, h_n) \in S_1''$. 
If $g+h > a$, then $g+h \in S_1' \subseteq S_1$. We cannot have $g+h=a$ since this would contradict the definition of $S_1''$. Hence, suppose $g+h < a$ and let us show that $g+h \in S_1''$.
If by way of contradiction $a-(g+h) \in S_2 \cup \ldots \cup S_n$, we would have $$a-h= g + (a-g-h) \in S_j + (S_2 \cup \ldots \cup S_n) \subseteq (S_2 \cup \ldots \cup S_n). $$ Also this contradicts the definition of $S_1''$ and shows that $S$ is a monoid.

Set $D: = [[K^S]]$. By construction $X^{a^2}, X^{a^3}$ are in $D$, hence $X^{-a}$ is in the quotient field of $D$. For every $g \in G_{\geq 0}$, $X^{g} = X^{a+g} X^{-a}$. This implies that $D$ has the same quotient field as $V$.
Using Theorem \ref{maxexclseries}, it is straightforward to check that $D$ is maximal with respect to excluding the element $X^a$.
\end{proof}

\begin{remark}
\label{remS1}
For the purpose of constructing examples, the conditions defining the set $S_1^*$ in the above proof are not very explicit. A good way to satisfy the third condition is assuming that, if $g_1 >0$, for every $g_2, \ldots, g_n$, the element $g=(g_1, g_2, \ldots, g_n) \in S_1^*$ if and only if $(g_1, 0, \ldots, 0) \in S_1^*$. For the first two conditions, one can choose a set such that the projection on the first component behaves as the set described in Example \ref{ex1}. An easier assumption, which can still produce many nice examples, is the following: set $G_1 = \Z$ and $a=(1, 0, \ldots, 0)$. In this way the set $S_1^* = \emptyset$ and the proof of Lemma \ref{importantlemma} can be simplified. However, there are other possible choices to construct a set $S_1^*$ satisfying the required conditions.
\end{remark}

We now prove the main theorem concerning the rings which occur as integral closure of a maximal excluding domain of the form $D= [[K^S]]$. 



The notation for the next theorem is slightly different from the one used until now. Let $n >1$ and let $G' \cong G_2 \oplus \ldots \oplus G_n$ be a totally ordered abelian group.
Set $V':=[[K^{G'_{\geq 0}}]]$.
Let $H$ be any submonoid of $G'_{\geq 0}$ such that $\overline{H}=H$ and suppose that $A =[[K^H]] \subseteq V' $ is integrally closed. We do not require $A$ and $V'$ to have the same quotient field. Given another totally ordered abelian group $G_1$ of rank one, set $G:= G_1 \oplus G'. $
Define $V:=[[K^{G_{\geq 0}}]]$. Let $W$ be the rank one valuation overring of $V$ and $\kappa$ be the residue field of $W$.

\begin{theorem}
\label{wildintegralclosure} 
With the notation and the assumptions stated above, there exists a submonoid $S$ of $G_{\geq 0}$ such that: 
\begin{enumerate}
\item[(i)] The ring $D =[[K^S]] \subseteq V $ is maximal excluding and has the same quotient field as $V$.
\item[(ii)] The integral closure $\overline{D} $ of $D$ occurs as the pullback in the following diagram:
\begin{center}
\begin{tikzcd}
  \overline{D} \arrow[rightarrow]{r}\arrow[hookrightarrow]{d} 
  &  A \arrow[hookrightarrow]{d}
  \\ 
   W \arrow{r}
 &  \kappa
\end{tikzcd}
\end{center}
\end{enumerate}
Conversely, if $D=[[K^{S'}]]\subseteq V $ is a maximal excluding domain having the same quotient field as $V$, then the integral closure of $D$ is the pullback of some integrally closed local domain $A'$ with respect to the quotient map $W \to \kappa$.
\end{theorem}

\begin{proof}
First notice that, by quotienting $G$ with respect to $G_1$, the quotient field of $V'$ can be identified with $\kappa$ and $V$ is the pullback of $V'$ with respect to the map $W \to \kappa$. Hence $A$ can be embedded in $\kappa$.

For $i=0, \ldots, n$, define $\widehat{G_i} $ and $\widehat{G'_i} $ as above in this section. 
Decompose $H$ as the union $H= H_0 \cup \bigcup_{i=2}^n H_i$, where $H_i = H \cap \widehat{G'_i}$. Since for $i \geq 2$ there exists a bijective map between $\widehat{G_i} $ and $\widehat{G'_i} $, we can identify $H_i$ with a subsemigroup of $\widehat{G_i} $.

To construct $S$, set $S_0= \lbrace \boldsymbol{0} \rbrace$ and for $i=2, \ldots, n$, define $S_i = H_i$. 
Since $H$ is a monoid, the semigroups $S_2, \ldots, S_n$ defined in this way obviously satisfy the hypothesis of Lemma \ref{importantlemma}. Use Lemma \ref{importantlemma} to produce $S_1 \subseteq \widehat{G_1}$ such that $S := \bigcup_{i=0}^n S_i$ is a monoid and the ring $D= [[K^S]]$ satisfies condition $(i)$. 

By construction, the ring $D$ is maximal with respect to excluding an element $a \in \widehat{G_1}$ and
 the set $ \lbrace g \in \widehat{G_1} \, | \, g > a \rbrace \subseteq S $. Given $g \in \widehat{G_1}$, since $G_1$ has rank one, there exists $n \geq 1$ such that $ng > a$. Thus $\overline{S_1}= \widehat{G_1}$. Clearly, $\overline{S}= \widehat{G_1} \cup H$ and $[[K^{\overline{S}}]]$ is integrally closed by standard properties of pullback diagrams.
 By Proposition \ref{rootclosure2}, $\overline{D}= [[K^{\overline{S}}]]$. Let $\pi: W \to \kappa$ be the canonical quotient map. We show that $\overline{D} = \pi^{-1}(A)$. Observe that, given $g \in G$, $\pi(X^g)=0$ if and only if $g_1 > 0$ and if $g_1 = 0$, $\pi(X^g)= X^{g'}$ where $g'$ is the component of $g$ in $G'$. It easily follows that $\pi(\overline{D}) = A$. Since $\mbox{ker}(\pi) \subseteq \overline{D}$, we obtain $\overline{D} = \pi^{-1}(A)$ and finally prove $(ii)$.
 
 Conversely, if $D=[[K^{S'}]]\subseteq V $ is a maximal excluding domain having the same quotient field as $V$, by Remark \ref{remarkpullback}, we can assume $D$ maximal with respect to excluding an element $X^a$ with $a \in \widehat{G_1}$.
 Using Theorem \ref{maxexclseries} and the same argument as above we get $\overline{S_1'} = \widehat{G_1}$.
  The thesis now follows setting $H':= \bigcup_{i=2}^n \overline{S_i} \cup \lbrace \boldsymbol{0} \rbrace$ and letting
  $A'$ be the integral closure of $[[K^{H'}]]$.
\end{proof}

The next example shows that both $D$ and its integral closure can have infinitely many incomparable prime ideals.

\begin{example}
\label{ex4}
Using the notation of Theorem \ref{wildintegralclosure}, let $G'= G_2 = \Z[\sqrt{2}]$ and $G_1= \Z$. Choose
$$H = \lbrace a+b\sqrt{2}, \, | \, a, b \geq 0 \rbrace \subseteq G'_{\geq 0}. $$ We notice that $H= \overline{H}$.
Indeed, if for some $n \geq 1$ and $a,b,c,d \in \Z$, $n(c+d\sqrt{2}) = a+\sqrt{2}b \in H$, using the irrationality of $ \sqrt{2} $, we must have $nc=a$ and $nd=b$, forcing $c,d \geq 0$.

Define $X$ to be the monomial element of $V$ of value $(0,1)$ and $Z$ to be the monomial element of $V$ of value $(0,\sqrt{2})$. The ring $A =[[K^H]] \cong K[[X,Z]]$ is a regular local ring of dimension 2. The quotient field of $A$ is properly contained in the quotient field of $V' =[[K^{G'}]]$. 

Now construct $S$ as in Theorem \ref{wildintegralclosure}. Looking at the notation of Lemma \ref{importantlemma} and at Remark \ref{remS1}, we notice that, if we choose $a=(1,0)$, we get $S_1^*= \emptyset$. Therefore Lemma \ref{importantlemma} gives a precise way to construct $S$ that does not require further choices. Let $Y$ be the monomial element of $V$ of value $(1,0)$. The ring $D = [[K^S]]$ is maximal with respect to excluding $Y$. 
The integral closure $\overline{D}$ is the pullback of $A$ with respect to the map $W \to \kappa$. 
Hence, $\overline{D}$ is local and has infinitely many incomparable prime ideals (of height 2). Therefore it can be neither a PVD nor Pr\"{u}fer.
In this example we can identify $W$ with $K((X,Z))[[Y]]$.
Then, if $\m_W$ is the maximal ideal of $W$, we have $\overline{D} \cong K[[X,Z]] + \m_V$. 

The prime ideals of height 2 of $\overline{D}$ are principal, generated by elements that are also elements of $D$ as a consequence of the way $S$ is defined in Theorem \ref{wildintegralclosure}. For this reason, they all contract to distinct prime ideals of $D$. It follows that also $D$ has infinitely many incomparable prime ideals. 
\end{example}



\section{Intersections of Nagata extensions of overrings}
In this section we consider intersections of Nagata rings in order to extend the concept of Kronecker function rings to non-integrally closed domains. For simplicity we suppose all the rings in this section to have finite Krull dimension. In particular every integrally closed maximal excluding domain will be a valuation domain.

Let $D$ be an integral domain. We say that a collection of overrings $\mathcal{F}= \lbrace D_i \rbrace_{i \in \Lambda} $ is a defining family for $D$ if $ \bigcap_{D_i \in \mathcal{F}} D_i = D. $ From what observed in Section 2, every integral domain always admits a defining family formed by maximal excluding overrings. Moreover we can always consider defining families of $D$ where all the non-integrally closed rings do not contain the integral closure of $D$. Indeed $D$ can be always express as $$ D = \overline{D} \cap \bigcap_{\alpha \in \mathcal{Q}(D) \setminus \overline{D}} A_{\alpha} $$ where $A_{\alpha}$ is a non-integrally closed overring of $D$ maximal with respect to excluding the element $\alpha$.

\begin{definition}
\label{defkron}
Let $D$ be an integral domain and let $\mathcal{F}$ be a defining family of $D$. We say that the ring 
$$  Kr(D)^{\mathcal{F}}= \bigcap_{D_i \in \mathcal{F}} D_i(t) $$ is the \it Kronecker function ring \rm of $D$ with respect to $\mathcal{F}$.
\end{definition}

When $D$ is integrally closed and the family $\mathcal{F}$ is contained in Zar$(D)$, the ring $Kr(D)^{\mathcal{F}}$ is a classical Kronecker function ring. We already recalled in Section 2, that an integrally closed domain is Pr\"{u}fer if and only if its Kronecker function ring $Kr(D)$ 
is equal to the Nagata ring $D(t)$ (in this case the Nagata ring is the unique Kronecker function ring).
Also if $D$ is a maximal excluding domain, since the only defining family of $D$ is $\lbrace D \rbrace$, we get that the only Kronecker function ring, according to Definition \ref{defkron} is $D(t)$. However, in general Pr\"{u}fer domains and maximal excluding domains may behave differently with respect to intersecting Nagata rings of overrings.
For a Pr\"{u}fer domain $D$, the operation of Nagata ring extension, mapping an overring $A$ to its Nagata ring $A(t)$, commutes with arbitrary intersection of overrings. Conversely, one may observe that if the integral closure of an arbitrary integral domain is not Pr\"{u}fer, then the operation of Nagata ring extension does not commute with intersection for some collections of overrings. This is due to the fact that the Kronecker function ring of an integrally closed domain that is not Pr\"{u}fer is always strictly bigger than the Nagata ring. 

We dedicate this section to investigating whether an intersection of Nagata rings is a Nagata ring. The next section is dedicated to constructing relevant families of non-integrally closed Kronecker function rings that behave similarly to the classical integrally closed ones.
The following question arises naturally:

\begin{question}
Let $D$ be an integral domain and suppose that $\overline{D}$ is a Pr\"{u}fer domain. If $\mathcal{F}$ is an arbitrary collection of overrings of $D$, then is $\bigcap_{A \in \mathcal{F}}A(t)= (\bigcap_{A \in \mathcal{F}}A) (t)$? In particular, is $\overline{D}$ a Pr\"{u}fer domain if and only if Nagata ring extension commutes with intersection for every collections of overrings of $D$?
\end{question}

A positive answer to the above question would imply that all the Kronecker function rings of a domain $D$ such that $\overline{D}$ is Pr\"{u}fer coincide with $D(t)$. We are able to give a positive answer to this question in Theorem \ref{semilocalprufer} in the case $\overline{D}$ is a semilocal Pr\"{u}fer domain of finite dimension and the residue fields of $D$ are sufficiently large. Our result implies a positive answer in the case $D$ has finite dimension and the integral closure of $D$ is a valuation ring or the intersection of two valuation rings, without other assumptions on the residue fields. We start by proving a lemma, which points out a condition for an intersection of local domains to be local. Given two (non-necessarily local) integral domains $ A \subseteq B $, we say that $B$ \it dominates \rm $A$ if every maximal ideal of $A$ is contained in some maximal ideal of $B$ and every maximal ideal of $B$ contracts to some maximal ideal of $A$. If $A$ and $B$ are local this means simply that $\m_{A} \subseteq \m_{B}$.




\begin{lemma}
\label{intersect1}
Let $D$ be an integral domain. Let
$\lbrace D_i \rbrace_{i \in \Lambda} $ be a collection of local overrings of $D$ such that $D= \bigcap_{i} D_i$. Suppose that there exists a valuation overring $V$ that contains all the rings $D_i$ and dominates all of them, except at most one. Then $D$ is local.
\end{lemma}

\begin{proof}
Call $D_1$ the ring in the family $\mathcal{F} $ possibly not dominated by $V$.
Pick two non-units $x, y \in D$. Then $x, y \in \bigcup_{i \in \Lambda} (\m_{D_i} \cap D)\subseteq (\m_{D_1} \cap D) \cup (\m_V \cap D) = \m_{D_1} \cap D. $ It follows that $x+y \in \m_{D_1} \cap D$ is not a unit in $D$. Notice also that if $D_1$ is dominated by $V$, then also $D$ is dominated by $V$.
\end{proof}

The main assumption that we need to prove Theorem \ref{semilocalprufer} is related to the cardinalities of the residue fields of an integral domain $D$. Notice that $D$ contains $s$ units $u_1, \ldots, u_s$ such that $u_i-u_j$ is still a unit for every $i,j$ if and only if all the residue fields of $D$ have cardinality at least $ s+1$. For domains containing an infinite field or containing $\Z(t)$, Theorem \ref{semilocalprufer} can be proved without further assumptions.



\begin{lemma}
\label{semilocalunits}
Let $D$ be a semilocal domain with $s$ maximal ideals. Suppose that all the residue fields of $D$ have cardinality at least $ s$.
Fix a set of units $U \subseteq D$, of cardinality at least $s-1$ and such that $u_i-u_j$ is a unit for every distinct $u_i, u_j \in U$. 
Let $x_1, \ldots, x_n \in D$ be elements such that $(x_1, \ldots, x_n)=D$.
Then there exists a unit of $D$ of the form $c = \sum_{k=1}^n \delta_k x_k $ with $\delta_k \in U \cup \lbrace 0 \rbrace$ for every $k$.
\end{lemma}

\begin{proof}
We prove it by induction on $n$. 
In the case $n=1$, we clearly must have that $x_1$ is a unit, and thus $\delta_1 x_1$ is a unit for every $\delta_1 \in U$. 
Thus assume the thesis true for every set of at most $n-1$ elements. 
Therefore if after reordering the generators $(x_1, \ldots, x_{n-1})=D$, we conclude by applying the inductive hypothesis and setting $\delta_n = 0$. 

Call $\m_1, \ldots, \m_s$ the maximal ideals of $D$. Possibly relabeling, suppose that $x_1, \ldots, x_{n-1} \in \m_s$, and there exist $1 \leq s' < s$ such that $x_n  \in (\m_1 \cap \ldots \cap \m_{s'}) \setminus (\m_{s'+1} \cup \ldots \cup \m_s)$. The ideal generated by $x_1, \ldots, x_{n-1}$ cannot be contained in any of the maximal ideals $ \m_1, \ldots, \m_{s'} $. Consider the ring $A= S^{-1}D$ where $S = D \setminus (\m_1 \cup \ldots \cup \m_{s'})$. The ring $A$ is semilocal and contains $U$. Moreover $(x_1, \ldots, x_{n-1})A=A$. By inductive hypothesis there exists a unit of $A$ of the form $c' = \sum_{k=1}^{n-1} \delta_k x_k $ with $\delta_k \in U \cup \lbrace 0 \rbrace$. Observing that $c' \in D$, we obtain $c' \not \in \m_1 \cup \ldots \cup \m_{s'}$. 

Take now two elements of the form $d_1= c'+ux_n$, $d_2 =c'+vx_n $ with $u,v \in U$, $u \neq v$. We claim that $d_1, d_2$ cannot be contained in a common maximal ideal. Indeed if $d_1, d_2 \in \m_i$, then we would have $d_1 - d_2 = (u-v)x_n \in \m_i$ and $vd_1-ud_2= (v-u)c' \in \m_i$, implying that $c', x_n \in \m_i$, since $u-v$ is a unit. This is a contradiction with the choice of $c'$. Moreover any element of the form $c'+ux_n$ cannot be in a maximal ideal containing either $c'$ or $x_n$.
Using now that $U$ contains at least $s-1$ elements and $D$ has only $s$ maximal ideals we can find an element of the form $c' + \delta_n x_n$ which is a unit in $D$.
\end{proof}

\begin{prop}
\label{intersectionnagatarings}
Let $D$ be an integral domain with quotient field $K$. Let
$\mathcal{F}= \lbrace D_i \rbrace_{i \in \Lambda} $ be a defining family of $D$.
Suppose that all the rings in $\mathcal{F} $ are contained in a common overring $T$ and one ring $D_1 \in \mathcal{F}$ is semilocal with $s$ maximal ideals. Also suppose that all the residue fields of $D$ have cardinality at least $s$ and $T$ dominates $D_i$ for every $i$, except possibly $D_1$. 
Then $D(t)= \bigcap_{i} D_i(t).$
\end{prop}

\begin{proof}
As in Lemma \ref{semilocalunits}, we can fix a set of units $U \subseteq D$, of cardinality at least $s-1$ such that $q_i-q_j$ is a unit for every distinct $q_i, q_j \in U$. Clearly the set $U$ is also contained in $D_1$.

We always have the inclusion $D(t) \subseteq \bigcap_{i} D_i(t).$ To prove the opposite inclusion pick $\phi \in \bigcap_{i} D_i(t)$. Such an element $\phi$ can be always written as $\frac{f}{g}$ with $f,g \in D[t]$. Say that $f= \sum_{k=0}^m a_k t^k$ and $g= \sum_{k=0}^n b_k t^k$. 

Since $\phi \in \bigcap_{i} D_i(t)$, for every $i$ we can find $x_i \in K$ such that $ \frac{a_k}{x_i}, \frac{b_k}{x_i} \in D_i $ for every $k$ and the ideal $\frac{1}{x_i}(b_0, \ldots, b_n)D_i$ is the unit ideal.
Hence we can choose an element $c_i \in (b_0, \ldots, b_n)D_i$ such that $c_i= x_iu_i$ with $u_i$ a unit in $D_i$. In particular $x_i \in D_i$. For $j \neq i$, since $ \frac{b_k}{x_j} \in D_j $ for every $k$, we get $ \frac{c_i}{x_j}= \frac{x_i}{x_j}u_i \in T $. Similarly, $ \frac{c_j}{x_i}= \frac{x_j}{x_i}u_j \in T $ and, since $u_i, u_j$ are units in $T$, we obtain that also $ \frac{x_i}{x_j} $ is a unit in $T$ for every $i,j$.

Let us now consider the semilocal domain $D_1$. The ideal $ \frac{1}{x_1}(b_0, \ldots, b_n)D_1 = D_1 $. By Lemma \ref{semilocalunits} we can choose the element $c_1 = \sum_{l=0}^n \delta_l b_l$ with all $\delta_l \in U \cup \lbrace 0 \rbrace \subseteq D$. Thus $c_1 \in D$. For every $i$, $\frac{c_1}{x_i} \in D_i$ and from what said above $\frac{c_1}{x_i}= \frac{x_1}{x_i}u_1$ is a unit of $T$. Since, for $i \neq 1$, the overring $T$ dominates $D_i$, then $\frac{c_1}{x_i}$ has to be a unit of $D_i$. It follows that for every $i$ we can write
$c_1 = x_iw_i$ with $x_i, w_i \in D_i$ and $w_i$ a unit in $D_i$ (for $i=1$, set $w_1= u_1$). Hence, $ \frac{a_k}{c_1}, \frac{b_k}{c_1} \in \bigcap_i D_i = D $. This implies that $ \frac{f}{c_1}, \frac{g}{c_1} \in D[t] $ and, since $c_1 \in (b_0, \ldots, b_n)D$ we get that the ideal of $D$ generated by the coefficients of $\frac{g}{c_1}$ coincides with $D$. This implies $ \phi= \frac{f}{g}\frac{c_1}{c_1} \in D(t) $.
\end{proof}

\begin{theorem}
\label{semilocalprufer}
Let $D$ be an integral domain such that $\overline{D}$ is a semilocal Pr\"{u}fer domain of finite dimension with $s$ maximal ideals. Suppose that all the residue fields of $D$ have cardinality at least $s$. 
Let $\mathcal{F}$ be any collection of local overrings of $D$ such that $\bigcap_{A \in \mathcal{F}} A = D$. Then $\bigcap_{A \in \mathcal{F}} A(t) = D(t)$.
\end{theorem}

\begin{proof}
The fact that $\overline{D}$ is a semilocal Pr\"{u}fer domain of finite dimension is equivalent to saying that $D$ has only finitely many valuation overrings, that we call $V_1, \ldots, V_n$.
For $i= 1, \ldots, n$, set $$\mathcal{F}_i:= \lbrace A \in \mathcal{F} \, : \, A \mbox{ is dominated by } V_i  \rbrace.$$ Observe that every ring in $\mathcal{F}$ has to belong to some $\mathcal{F}_i$.
Let $A_i:= \bigcap_{A \in \mathcal{F}_i} A $. Notice that by Lemma \ref{intersect1}, $A_i$ is local and dominated by $V_i$. By Proposition \ref{intersectionnagatarings}, $\bigcap_{A \in \mathcal{F}_i} A(t) = A_i(t)$. 

Since $\overline{D}$ is a Pr\"{u}fer domain, its Nagata ring coincides with the Kronecker function ring and therefore $\overline{D}(t) = V_1(t) \cap \ldots \cap V_n(t)$. Hence we have the equality $$ \bigcap_{A \in \mathcal{F}} A(t) = \bigcap_{i=1}^n A_i(t)  = \left( \bigcap_{i=1}^n A_i(t) \right) \cap \overline{D}(t)=    \bigcap_{i=1}^n (A_i(t) \cap \overline{D}(t)). $$
Set $D_i:= A_i \cap \overline{D} $. There are exactly $s$ minimal valuation overrings of $\overline{D}$ and $A_i$ is dominated by an overring of one of them, hence 
$D_i$ is a proper intersection of $s$ local rings, therefore is semilocal with at most $s$ maximal ideals. Moreover, $D = \bigcap_{i=1}^n D_i$ and all the $D_i$ are dominated by $\overline{D}$ (clearly $\overline{D}$ is the integral closure of all of them). Again by Proposition \ref{intersectionnagatarings}, $D(t)= \bigcap_{i=1}^n D_i(t)$. To conclude we only need to prove that $D_i(t)= (A_i \cap \overline{D})(t)= A_i(t) \cap \overline{D}(t)$. This follows by Proposition \ref{intersectionnagatarings} since both $A_i$ and $\overline{D}$ are contained in $V_i$, $A_i$ is dominated by $V_i$, and $\overline{D}$ is semilocal.
\end{proof}

 Notice that the above theorem implies that if $\overline{D}$ is a Pr\"{u}fer domain of finite dimension with at most 2 maximal ideals, then for every defining family $\mathcal{F}$, we have $Kr^{\mathcal{F}}(D)= D(t)$.
 We ask whether the same result of Theorem \ref{semilocalprufer} holds in general removing the assumptions on the residue fields of $D$.

We pass now to study the case of an intersection of a semilocal domain $A$ and a valuation domain $V$, having the same quotient field. This case is relevant for the results in Section 5. Again, in the case $A$ is local, the thesis of the next theorem holds without assumptions on the residue fields.

\begin{theorem}
\label{VcapA}
Let $V$ be a valuation domain of finite dimension and let $A$ be a semilocal domain such that $\mathcal{Q}(A)= \mathcal{Q}(V)= \mathcal{Q}(V \cap A)$. Assume that $A$ has $s$ maximal ideals $\m_1, \ldots, \m_s$ and all the residue fields of $A \cap V$ have cardinality at least $s+1$.
The following assertions hold:
\begin{enumerate}
\item[(i)]  $A \cap V$ is local if and only if $A$ is local and is dominated by an overring of $V$. In this case $\m_A$ is a common prime ideal of $A$ and $A \cap V$, $A = (A \cap V)_{\m_A}$, and $ \frac{A \cap V}{\m_A} $ is a valuation domain with quotient field $ \frac{A}{\m_A}$.
\item[(ii)] If $A \cap V$ is not local, then $A_{\m_i} = (A \cap V)_{\m_i \cap V}$ for every $i=1, \ldots, s$.
\item[(iii)] If $A$ and $ V$ have no common proper overrings, then $V = (A \cap V)_{\m_V \cap A}.$
\item[(iv)] If $A \cap V$ is not local and $A \subseteq V_Q$ for some prime ideal $Q $ of $V$, maximal with respect to this property, then $ (A \cap V)_{\m_V \cap A} = A_{Q \cap A} \cap V.$
\item[(v)] $(A \cap V)(t)= A(t) \cap V(t)$.
\end{enumerate}
\end{theorem}

\begin{proof}
We can assume $A \nsubseteq V$, otherwise all the results are obvious. The assumption on the residue fields of $A \cap V$ implies that if $z \in \m_i$ for some $i$, then there exists a unit $q \in A \cap V$ such that $z + q$ is a unit in $A$.  \\
(i) 
If $A$ is local and dominated by an overring of $V$, the implication follows by Lemma \ref{intersect1}. Conversely if $A \cap V$ is local we claim that $\m_i \subseteq V$ for every $i$. Indeed, if there exists $z \in \m_i \setminus V$, we get $z^{-1} \in \m_V \setminus A$. Choose a unit $q \in A \cap V$ such that $z + q$ is a unit in $A$. Observing that $$  \frac{z}{q+z} = \frac{1}{qz^{-1}+1} \in \m_i \cap (V \setminus \m_V), \quad \frac{qz^{-1}}{1+qz^{-1}} = \frac{q}{z+q} \in \m_V \cap (A \setminus \bigcup_{j=1}^s\m_j), $$ we find that $\frac{z}{q+z}$ and $\frac{q}{z+q}$ are two non-units of $A \cap V$ whose sum is a unit. This is a contradiction. Hence we have $\m_i \subseteq V$ for every $i$. Pick an element $s \in A \setminus V$. For every $n \geq 1$ and every $x \in \bigcup_{j=1}^s\m_j$ we get $s^nx \in V$, showing that the radical of $xV$ is properly contained in the radical of $s^{-1}V$. Therefore, there exists a non-maximal prime ideal $Q$ of $V$ such that $\bigcup_{j=1}^s\m_j \subseteq Q= QV_Q$ and $A \subseteq V_Q$. In particular the unique maximal ideal of $A$ is $\m_A= Q \cap A$ and $A$ is dominated by $V_Q$. 
Using that $\m_A \subseteq V$, we obtain that $\m_A$ is a prime ideal of $A \cap V$. For $a \in A \setminus V$, we have $a^{-1} \in (A \cap V) \setminus \m_A$, thus $A = (A \cap V)_{\m_A}$. If $x,y \in (A \cap V) \setminus \m_A $, then the fractions $ \frac{x}{y}, \frac{y}{x} $ are in $A$ and at least one of them is in $V$, implying that $ \frac{A \cap V}{\m_A} $ is a valuation domain.  \\
(ii) It is sufficient to show that $A\subseteq (A \cap V)_{\m_i \cap V}$. For $z \in A \setminus V$, if $z^{-1} \in A$ then $ z= \frac{1}{z^{-1}} \in (A \cap V)_{\m_i \cap V}. $ If $z $ is not a unit, pick again a unit $q \in A \cap V $ such that $q+z$ is a unit in $A$. Set $u = \frac{1}{q+z}$ and observe that $u$ is a unit of $A$ such that $uz = \frac{1}{qz^{-1}+1} \in V$ and $u \in V$. Hence $ z = \frac{zu}{u} \in  (A \cap V)_{\m_i \cap V}$. \\
(iii) For $x \in V \setminus A$, if $x^{-1} \in A$, we conclude that $x \in (A \cap V)_{\m_V \cap A}$ as in item (ii) by choosing $u= \frac{1}{1+qx}$ with $q$ a unit in $A \cap V$ such that $q+x^{-1}$ is a unit in $A$. Suppose that $x, x^{-1} \not \in A$ and write $x=\frac{a}{b}$ with $a, b $ both non-units in $A$ and $v(a) \geq v(b)$. 
We want to find a unit $u$ of $A$ such that $v(u)= - v(b)$. In this way we can write $ x = \frac{au}{bu} \in  (A \cap V)_{\m_V \cap A}$.

If there exists a unit $t$ of $A$ such that $v(t) > v(b)$, it is sufficient to set $u = \frac{1}{t+qb}$ where $q \in A \cap V $ is a unit such that $t+qb$ is a unit in $A$. Suppose by way of contradiction that there are no units of $A$ with value larger than or equal $v(b)$. This implies that $v(b) > 0$ (because $1$ is a unit of both $A$ and $V$). By assumption on the residue fields, every element of $A$ can be expressed as the sum of at most two units, thus, since $A \nsubseteq V$, there exists some unit $s$ of $A$ such that $v(s) < 0$.
For every such $s$, we get that the radical of $s^{-1}V$ properly contains the radical of $bV$. In particular all the units of $A$ are contained in a proper overring $V_Q$ of $V$. 
This forces $A \subseteq V_Q$. This is a contradiction. \\
(iv) Clearly $V_Q$ dominates $A_{Q \cap A}$. By Lemma \ref{intersect1}, $A_{Q \cap A} \cap V$ is local and contains $A \cap V$. In particular it contains the localization $(A \cap V)_{\m_V \cap A}. $ For $x \in A_{Q \cap A} \cap V$ we argue as we did for item (iii) to show that $ x = \frac{au}{bu} \in  (A \cap V)_{\m_V \cap A}$. Now we can choose $a, b $ non-units in $A$ such that $b \not \in Q \cap A$. By the choice of $Q$, there exists some unit $s$ of $A$ such that the radical of $sV$ is the unique prime ideal $Q'$ of $V$ such that $\hgt(Q') = \hgt(Q)+1$. Hence, there exists a unit $t$ of $A$ such that $v(t) > v(b)$. We can conclude using the same proof as for item (iii). \\
(v) We consider different cases. If $A \cap V$ is local, by item (i), $A$ is dominated by a valuation overring of $V$. The thesis follows now by Proposition \ref{intersectionnagatarings}.
If $A \cap V$ is not local and $A$ and $V$ have no common overrings, by items (ii)-(iii), all the rings $A_{\m_i}$ and $V$ are precisely the localizations of $A \cap V$ at the maximal ideals. The thesis follows since Nagata ring extension commutes with localizations at the maximal ideals.
If instead we are in the situation described by item (iv), we use item (ii) and localization to say that $(A \cap V)(t)= A(t) \cap (A_{Q \cap A} \cap V)(t)$. Using that $ A_{Q \cap A} \cap V $ is local we can split further $(A_{Q \cap A} \cap V)(t) = A_{Q \cap A}(t) \cap V(t)$ and conclude since $A \subseteq A_{Q \cap A}$.
\end{proof}

Also about Theorem \ref{VcapA}, we ask whether the same results hold removing the assumption on the cardinality of the residue fields.
We conclude this section by showing that under some mild conditions, Nagata ring extension does not commute with intersection also in the non-integrally closed case. Various constructions of this kind will be analyzed in the next section (see Constructions \ref{construction} and \ref{construction2}).

\begin{prop}
\label{notnagata}
Let $D$ be a local integral domain and let
$\mathcal{F}= \lbrace D_i \rbrace_{i \in \Lambda} $ be a defining family of $D$. 
 If there exists $z \in \mathcal{Q}(D)$ such that $z, z^{-1} \not \in D$ and for every $i$, $z \in D_i$ or $z^{-1} \in D_i$, then $D(t) \neq \bigcap_{i \in \Lambda} D_i(t).$
 In particular if $D = T \cap A$ with $T$ integrally closed, and there exists $z \in A $ such that $z, z^{-1} \not \in D$. 
Then $ D(t) \neq \bigcap_{V \supseteq T} V(t) \cap A(t)$.
\end{prop}

\begin{proof}
 Consider 
 $$ \phi= \frac{1}{t+z} = \frac{z^{-1}}{z^{-1}t+1}.  $$ Clearly $\phi \in \bigcap_{i \in \Lambda} D_i(t)$. 
 Let us show that $\phi \not \in D(t)$. We can multiply by a common factor to have both numerator and denominator of $\phi$ inside $D$. Hence pick any $d \in D$ such that $dz \in D$ and write $\phi= \frac{d}{dt+dz}$. If $\phi \in D(t)$, since $D$ is local we would have that either $d$ or $dz$ is a unit in $D$. But $d$ cannot be a unit, since $dz \in D$ and $z \not \in D$. If $dz$ were a unit, we would have $z^{-1} = d (dz)^{-1} \in D$. In any case this leads to a
  contradiction.
\end{proof}

\begin{remark}
For a local domain $D$ that is not a valuation ring and an overring $A$, 
the existence of an element $z \in A$ such that $z, z^{-1} \not \in D$ is usually satisfied. Indeed, such an element $z$ does not exist if and only if $A = D_Q$ for a prime ideal $Q$ of $D$ such that $Q=QD_Q$, and $\frac{D}{Q}$ is a valuation ring. For this pick $z \in A$ not a unit and observe that $z \in \m_D$. Thus for $x, y$ not units in $A$, if $x+y$ was a unit in $A$, we would get $1 = \frac{x}{x+y}+\frac{y}{x+y} \in \m_D$, which is impossible and implies that $A$ is local. 
The maximal ideal of $A$ is equal to a prime ideal $Q$ of $D$ which forces $Q=QD_Q$ and $A=D_Q$. Finally, if $u, v$ are element of $D \setminus Q$, then $\frac{u}{v}, \frac{v}{u}$ are units in $A$ and at least one of them has to be in $D$, implying that $\frac{D}{Q}$ is a valuation domain. Notice that by Theorem \ref{VcapA}-(i), a domain $D$ of this form may arise as intersection $D= A \cap V$ for some valuation ring $V$ which has an overring dominating $A$.
\end{remark}

The above remark suggests an example of local domain $D$ not maximal excluding, such that $\overline{D}$ is not Pr\"{u}fer, but still we have $Kr^{\mathcal{F}}(D)=D(t)$ for some defining family $\mathcal{F}$. However, Nagata ring extension clearly does not commute with intersection for all the overrings of $D$.

\begin{example}
\label{ex41}
Let $T$ be the maximal excluding generalized power series ring of Example \ref{ex3}, defined over the field $K= \mathbb{Q}$. Set $D= \Z_{(p)}+ \m_T$ for a prime number $p$. The integral closure of $D$ is $\overline{D}= \Z_{(p)}+ \m_{\overline{T}}$ and it is not Pr\"{u}fer since $\overline{T}$ is not. Also $D$ is not maximal excluding by Theorem \ref{pullbacksquare} because $T$ is not a valuation ring. We can write $D = \overline{D} \cap T$ and consider the defining family $\mathcal{F}= \Zar(D) \cup \lbrace T \rbrace$, setting $Kr^{\mathcal{F}}(D) = Kr(\overline{D}) \cap T(t)$. 
Let $V$ be a valuation overring of $D$ such that $\m_{V}=pV$. Observing that there are no rings properly contained between $D $ and $ T$ and $p^{-1} \in T \setminus V$, we get $V \cap T = D$ (of course $T$ is dominated by an overring of $V$, as expected by Theorem \ref{VcapA}-(i)). 
Now, Theorem \ref{main2}-(i) from the next section applies to this setting to show that the localization of $Kr^{\mathcal{F}}(D)$ at the center of $V(t)$ is equal to $D(t)$. Thus $Kr^{\mathcal{F}}(D) = D(t)$.
\end{example}

We leave another open question.

\begin{question}
Is it possible to find an integral domain $D$ not maximal excluding, such that $\overline{D}$ is not Pr\"{u}fer but $Kr^{\mathcal{F}}(D)=D(t)$ for every defining family $\mathcal{F}$ of $D$? Is this true for the ring $D$ defined in Example \ref{ex41}?
\end{question}



\section{Constructions of non-integrally closed Kronecker function rings}

In this section we construct non-integrally closed rings of the form $Kr^{\mathcal{F}}(D)$, according to the notation of Definition \ref{defkron}. The two main questions that we investigate for such rings are: understanding what the integral closure is and studying if they behave locally like classical Kronecker function rings, in the sense that the localizations at maximal ideals are Nagata ring extensions of some overring of the base ring $D$. We immediately observe that in general the integral closure of $Kr^{\mathcal{F}}(D)$ may not coincide with the Kronecker function ring $Kr(\overline{D})$. Indeed we have:

\begin{remark}
\label{rem51}
Let $D$ be an integral domain, $\mathcal{F}$ a defining family for $D$ and set $R= Kr^{\mathcal{F}}(D)$. Suppose that $ \overline{R}= Kr(\overline{D}). $ Then the integral closure of every $A \in \mathcal{F}$ is a Pr\"{u}fer domain. For this simply recall that $Kr(\overline{D})$ is Pr\"{u}fer and integral closure commutes with Nagata ring extension.
\end{remark}

In the following we first consider cases where the integral closure of $Kr^{\mathcal{F}}(D)$ is equal to $Kr(\overline{D})$. Then, recalling that the integral closure of a maximal excluding domain may not be a Pr\"{u}fer domain, we give in Theorem \ref{main2} and Example \ref{ex54}, examples such that the integral closure is a proper non-Pr\"{u}fer subring of $Kr(\overline{D})$.

Regarding the local behavior of $Kr^{\mathcal{F}}(D)$, we find that in all our constructions, the localizations at the maximal ideals are Nagata ring extensions of overrings of $D$. For this reason we leave the following general question for further research:

\begin{question}
Let $D$ be an integral domain and let $\mathcal{F}$ be any defining family of $D$. Let $\p$ be a maximal ideal of the ring $R=\bigcap_{A \in \mathcal{F}} A(t)$. Is $R_{\mathfrak{p}}= C(t)$ for some overring $C$ of $D$?
\end{question}

Our first construction is based on integral domains whose integral closure is obtained by adding the generators of a finite algebraic (Galois) field extension. We restrict to work with $D$ a local domain, since one can always reduce from the global case to the local one by localizing at each maximal ideal.

\begin{construction}
\label{construction} \rm
Let $D$ be a local domain with quotient field $F$. Let $F'$ be a subfield of $F$ such that $\mathbb{Q} \subseteq F'$ and $F$ is a finite Galois extension of $F'$, generated as $F'$-vector space by $\theta_1=1, \theta_2, \ldots, \theta_n$. Call $K$ the field generated over $\mathbb{Q}$ by $\theta_1, \theta_2, \ldots, \theta_n$ and all their conjugates with respect to the action of Gal$(F/F')$.
We assume any intermediate field extension $E/ E'$ such that $F' \subseteq E' \subseteq E \subseteq F$ to be a Galois extension generated by $\vartheta_1, \ldots, \vartheta_m$ and satisfying one of the following condition (see Examples \ref{ex51} and \ref{ex52} for examples): 
\begin{itemize}
\item[($\ast_1$)] The field $K$ 
is contained in $ \overline{D}$.
\item[($\ast_2$)] $\lbrace \vartheta_1, \ldots, \vartheta_m \rbrace = \lbrace 1, \vartheta, \vartheta^2, \ldots, \vartheta^{m-1} \rbrace$ where $\vartheta$ is a simple root of degree $m $ over $E'$ and $\overline{D}$ contains
$\xi \vartheta^m $ for every $m$-th root of unity $\xi$.
\end{itemize}
Suppose there exists a local integrally closed domain $D'$ containing $\Q$, having quotient field $F'$ such that $\m_{D'} \subseteq \m_{D}$, $D= D'[\m_D]_{\m_D}$, and
$\overline{D'}^F= \overline{D} = D[\theta_1, \ldots, \theta_n]$. 

Given a valuation overring $V$ of $D'$ such that $\m_V \supseteq \m_{D'}$, define $A_V:= V[\m_D]_{(\m_{V}, \m_D)}$. Clearly $D \subseteq \bigcap_{V \supseteq D'} A_V \subseteq \overline{D} $. Suppose that $\bigcap_{V \supseteq D'} A_V = D$, and in that case define $\mathcal{F}$ to be the defining family of $D$ containing all the maximal excluding overrings $A$ such that $A \supseteq A_V$ for some $V$ valuation overring of $D'$. Set $R:=Kr^{\mathcal{F}}(D)$.
\end{construction}

The main examples of rings $D$ that we can obtain from Construction \ref{construction} are certain kinds of $K+\m$ constructions and of algebras over fields, as shown in the next examples.

\begin{example}
\label{ex51}
Let $D'= K'+ \mathfrak{n}$ be an integrally closed domain with quotient field $F'$, such that $\mathbb{Q} \subseteq K'$. Let  $K$ be a finite Galois extension of $K'$ generated as $K'$-vector space by $\theta_1, \ldots, \theta_n $ and such that $K \cap F' = K'$. Call $F$ the finite Galois extension of $F'$ generated by the same elements. Then define $D= K' + \m$ with $\m = \sum_{i=1}^n \theta_i \mathfrak{n}, $ and observe that
$\overline{D}= \overline{D'}^F=  K+ \m$. 
Clearly the field extension $F/F'$ satisfies condition ($\ast_1$). By construction the rings $A_V$ do not contain any element of $K \setminus K'$. Therefore $\bigcap_{V \supseteq D'} A_V = D$.
\end{example}

\begin{example}
\label{ex52}
Consider a $K$-algebra of the form $D= \sum_{i=1}^c K[[x_i^{e_{1i}}, \ldots, x_i^{e_{l_ii}}]]$ (or the corresponding localized polynomial ring version), where $x_1, \ldots, x_c$ are indeterminates over $K$, the exponents $1 \leq e_{1i} < \ldots < e_{l_ii}$ generate a numerical semigroup, and
$K $ contains the $e_{1i}$-th cyclotomic field for all $i=1, \ldots, c$. In this case $\overline{D}= K[[x_1, \ldots, x_c]] $ and $D'= K[[x_1^{e_{11}}, \ldots, x_c^{e_{1c}}]]. $  The field extension $F/F'$ satisfies condition ($\ast_2$).
To see that $\bigcap_{V \supseteq D'} A_V = D$, observe that if $V$ is a valuation overring of $D'$ (with $\m_V \supseteq \m_{D'}$) such that $x_1^{e_{11}}$ has infinitely larger value than all the other $ x_i^{e_{1i}} $, then all the elements of the form $x_1^s f(x_2, \ldots, x_c) $ are not in $A_V$ if $s $ is not in the semigroup generated by $e_{11}, \ldots, e_{l_11}$. Similarly one can show that each element of $\overline{D} \setminus D$ does not belong to some ring $A_V$.
Simple examples that one may consider in this family are rings such as $K[[x^2,x^3,y]]$ and $K[[x^2,x^3,y^2,y^3]]$.
\end{example}

%



\begin{lemma}
\label{fieldextensions}
Take the assumptions and notations of Construction \ref{construction}. Let $V$ be a valuation overring of $D'$. 
Then any element $ \sum_{j=1}^n a_j \theta_j \in F$ with $a_j \in F'$ is integral over $V$ if and only if each summand $a_j\theta_j$ is integral over $V$. 
\end{lemma}

\begin{proof}
Let $\sigma_1, \ldots, \sigma_n$ be the elements of Gal$(F/F')$. By restricting each time to the automorphisms with fixed field $E'$, we can reduce to prove this result only for the main extension $F/ F'$, assuming that it satisfies either condition $(\ast_1)$ or condition $(\ast_2)$. 

Set $B_V:=\overline{V}^{F}.$
Let $\alpha=  \sum_{j=1}^n a_j \theta_j \in B_V$ with $a_j \in F'$. Since $F/ F'$ is a Galois extension, it is well-known that the integral closure of $V$ in $F$ is the intersection of all the extensions of $V$ to the field $F$, and such extensions are all conjugates by the elements of Gal$(F/F')$ (see for instance \cite[Section 3.2]{valuedfields}).  Therefore, also all the conjugates of $\alpha,$ $\sigma_i(\alpha)$ are elements of $ B_V. $ It follows that any linear combination of $\sigma_1(\alpha), \ldots, \sigma_n(\alpha)$ with coefficients in $B_V$ is in $ B_V. $ 
Now, set up a linear square system to express $a_j \theta_j$ as a linear combination of $\sigma_1(\alpha), \ldots, \sigma_n(\alpha)$. The equations we obtain are of the form $ a_j \theta_j = a_j \sum_{k=1}^n b_k \sigma_k(\theta_j)$ and $0=a_i\sum_{k=1}^n b_k \sigma_k(\theta_i)$ for $i \neq j$. These can be solved independently of $a_1, \ldots, a_n$ for a choice of coefficients $b_1, \ldots, b_n$ in the field $K$, which is generated over $\mathbb{Q}$ by the elements $\sigma_i(\theta_k)$. Assuming condition $(\ast_1)$, we have that $ K \subseteq \overline{D} \subseteq B_V $. The square matrix defining the linear system is $$ \bmatrix 
1 & 1 & \ldots & 1 \\ 
\theta_2 & \sigma_2(\theta_2) & \ldots & \sigma_n(\theta_2) \\
\vdots & \vdots & \ldots & \vdots \\
\theta_n & \sigma_2(\theta_n) & \ldots & \sigma_n(\theta_n)
\endbmatrix.
$$
The determinant of this matrix lives in $K$ and is equal to the square root of the discriminant of the field extension $F/F'$. This extension is separable, hence the discriminant is nonzero. Since $ K \subseteq B_V $, the system can be solved over $B_V$ by Cramer's rule yielding $a_j \theta_j \in B_V$ for every $j$.
In case condition $(\ast_2)$ is satisfied, the field $F$ is generated over $F'$ by simple roots, and by hypothesis we have $\theta_j= \theta_2^{j-1}$ for $j=1, \ldots, n$. Moreover, the $n$-th cyclotomic field $\mathbb{Q}(\xi)$ is contained in $ \overline{D}$, hence in $B_V$. To solve the same linear system as above with respect to $ a_j \theta_j $, we reduce to solve a linear square system in $\mathbb{Q}(\xi)$ whose matrix is a Vandermonde matrix in the entries $1, \xi, \ldots, \xi^{n-1}$. Since $\mathbb{Q}(\xi) \subseteq B_V$ we again obtain $a_j \theta_j \in B_V$ for every $j$.
\end{proof}

The next lemma gives us a sufficient condition to have a localization of an overring of $D(t)$ equal to the Nagata ring extension of an overring of $D$.

\begin{lemma}
\label{localizationlemma}
Let $D$ be an integral domain and let $A$ be a local overring. Suppose $R$ to be an integral domain such that $D(t) \subseteq R \subseteq A(t)$. 
Let $\mathfrak{p}= \mathfrak{m}_{A(t)} \cap R$ and suppose $A \subseteq R_{\mathfrak{p}}$. Then $R_{\mathfrak{p}} = A(t)$. 
\end{lemma}

\begin{proof}
By assumption $t \in R$, thus there are inclusions $A[t] \subseteq R_{\mathfrak{p}} \subseteq A(t)$.  
The ring $A(t)$ is the localization of $A[t]$ at the prime ideal $\mathfrak{m}_{A}[t] = \mathfrak{m}_{A(t)} \cap A[t]$. 
From the fact that $\mathfrak{p}R_{\mathfrak{p}} \subseteq \mathfrak{m}_{A(t)}$, it follows that $\mathfrak{q}:= \mathfrak{p}R_{\mathfrak{p}} \cap A[t] \subseteq \mathfrak{m}_{A}[t]$. This implies that $A(t) \subseteq (A[t])_{\mathfrak{q}} \subseteq R_{\mathfrak{p}}$. The thesis follows.
\end{proof}

\begin{theorem}
\label{main1}
Let $R$ be defined as in Construction \ref{construction}. Then $\overline{R}= Kr(\overline{D})$ and all the localizations of $R$ at the maximal ideals are of the form $R_{\p} = A_V(t)$ where $\p = \m_{A_V(t)} \cap R$.
\end{theorem}

\begin{proof}
By construction, for every valuation overring $V$ dominating $D'$, the integral closure of $A_V$ is $B_V:=\overline{V}^F$, which is the intersection of all the extensions of $V$ to $F$, of which there are finitely many. Hence $\overline{A_V}$ is a semilocal Pr\"{u}fer domain. 
By Theorem \ref{semilocalprufer}, Nagata ring extension commutes with intersection for all the overrings of $A_V$ (for this we use the assumption that $ \mathbb{Q} \subseteq D' \subseteq A_V$). It follows that $R = \bigcap_{V \supseteq D'} A_V(t)$ and therefore $$Kr(D')= \bigcap_{\footnotesize V \in \mbox{Zar}(D')} V(t) \subseteq R \subseteq Kr(\overline{D})= \bigcap_{\footnotesize V \in \mbox{Zar}(D')} B_V(t).$$ Pick now an element $\Phi \in Kr(\overline{D})$. The field extension $F(t)/F'(t) $ is clearly a finite Galois extension generated by $\theta_1, \ldots, \theta_n$ and satisfies one of the conditions $(\ast_1)$, $(\ast_2)$, since so does the field extension $F/F'. $ Hence, there exist $\phi_1(t), \ldots, \phi_n(t) \in F'(t)$ such that $$ \Phi= \sum_{j=1}^n \phi_j(t) \theta_j \in Kr(\overline{D})= \bigcap_{\footnotesize V \in \mbox{Zar}(D')} B_V(t).$$ By Proposition \ref{fieldextensions}, this yields $\phi_j(t) \theta_j \in Kr(\overline{D})$ for every $j=1, \ldots, n$. 
If now all $\theta_j$ are units in $\overline{D}$, this gives $$ \phi_j(t) \in \bigcap_{\footnotesize V \in \mbox{Zar}(D')} B_V(t) \cap F'(t)= \bigcap_{\footnotesize V \in \mbox{Zar}(D')} V(t) = Kr(D') \subseteq R$$ and shows that $\Phi$ is integral over $R$ since so are all $\theta_j$.
If instead all $\theta_j$ are simple roots, saying that $\theta_j^{e_j} \in D'$, we obtain $(\phi_j(t) \theta_j)^{e_j} \in Kr(\overline{D}) \cap F'(t) \subseteq R,$ again showing that $\Phi$ is integral over $R$, being a sum of integral elements.

Now, using that $\overline{R}= Kr(\overline{D})$, we obtain that all the maximal ideals of $R$ are centers of trivial extensions of valuation overrings of $D$. In particular, since $R \subseteq A_V(t) \subseteq \overline{A_V}(t)= B_V(t)$, every maximal ideal of $R$ is of the form $\p = \m_{A_V(t)} \cap R$. 
This proves the inclusions $R_{\mathfrak{p}} \subseteq A_V(t)$ and $ \mathfrak{p}R_{\mathfrak{p}} \subseteq \mathfrak{m}_{A_V(t)} $. We show now that $A_V \subseteq R_{\p}$. For this observe first that $\m_D \subseteq \m_{A_V} \cap R \subseteq \p$. Then pick $\alpha \in V$. The fraction $ \frac{1}{t+\alpha} \in Kr(D') \subseteq R $ but $v(\frac{1}{t+\alpha})=0$, hence $\frac{1}{t+\alpha} \not \in \m_{V(t)}= \m_{A_V(t)} \cap V(t) $. This implies $\frac{1}{t+\alpha} \not \in \p $, thus $\alpha \in R_{\p}$. It is clear that $\m_{V(t)} \subseteq \p$. It follows that $A_V \subseteq R_{\p}$. 
The thesis now follows by Lemma \ref{localizationlemma}.
\end{proof}

The above theorem shows that the localizations of $R$ at prime ideals coincide with the localizations of some Nagata ring extensions of overrings of $D$.
Essentially by the same proof we can obtain the same result also for some integral extensions of $R$. 

\begin{corollary}
\label{corcostruction}
Take the assumptions and notations of Construction \ref{construction}. For each ring $A_V$, let $C_V$ be a local integral extension. Let $T= \bigcap_{V \supseteq D'} C_V$ and suppose $T \neq \overline{D}$. Let 
$R^*= \bigcap_{V \supseteq D'} C_V(t)$. Then the ring $R^*= Kr^{\mathcal{F}^*}(T)$ for some defining family $\mathcal{F}^*$ of $T$ and the integral closure of $R^*$ is $Kr(\overline{D})$. Moreover, let $\p = \m_{C_V(t)} \cap R$. If $C_V \subseteq R^*_{\p}$, then $R^*_{\p} = C_V(t)$.
\end{corollary}

\begin{proof}
The integral closure of any of the rings $C_V$ is semilocal and Pr\"{u}fer. Hence, as in the proof of Theorem \ref{main1}, the family $\mathcal{F}^*$ can be chosen to be the family of all the maximal excluding overrings of the rings $C_V$. By definition we have inclusions $R \subseteq R^* \subseteq R \subseteq Kr(\overline{D})$, showing that $\overline{R^*}= \overline{R}= Kr(\overline{D})$. The last part about localizations follows by Lemma \ref{localizationlemma}.
\end{proof}

Corollary \ref{corcostruction} can be applied to the following situation. 

\begin{prop}
\label{propcostruction}
Let $D= K'+\m$ be constructed as in Example \ref{ex51}. 
Let $C_V= V[\sum_{i=1}^n \theta_i \m_{V}] $ and let $\p = \m_{C_V(t)} \cap R$. Then $C_V$ is local and $A_V \subseteq C_V \subseteq R^*_{\p}$. 
\end{prop}

\begin{proof}
From the inclusions $\m_D \subseteq \sum_{i=1}^n \theta_i \m_{D'} \subseteq \sum_{i=1}^n \theta_i \m_{V}$, it follows that $A_V \subseteq C_V$ and $\overline{C_V}= \overline{V}^F$. Let $\sigma_1, \ldots, \sigma_n$ be the elements of Gal$(F/F')$ with $\sigma_1$ equal to the identity.
To see that $C_V$ is local, pick $\alpha= a + \sum_{i=1}^n \theta_i b_i$ with $b_i \in \m_V$ and $a$ a unit in $V$.
Observe that $\sigma_j(\theta_i)b_i \in \m_VC_V \subseteq \mbox{Jac}(\overline{V}^F)$ for every $i,j$. In particular the product $\prod_{i=1}^n \sigma_i(\alpha) = a^n + z$, with $z \in \m_VC_V$, is a unit in $C_V$. It follows that $$ \frac{1}{\alpha}= \frac{\prod_{i=2}^n \sigma_i(\alpha)}{\prod_{i=1}^n \sigma_i(\alpha)} \in C_V. $$  
 
Now let us prove that $C_V \subseteq R^*_{\p}$. By the same argument used in the proof of Theorem \ref{main1}, we obtain $V \subseteq R^*_{\p}$. We need to show that $z\theta_i \in R^*_{\p}$ for $z \in \m_V$. For this consider the fraction $ \phi=\frac{1}{t+z\theta_i} $. For all the valuation overrings $V'$ of $D'$ containing $z$, we have $\phi \in C_V(t)$. For the valuation overrings $V'$ not containing $z$, we have $z^{-1} \in \m_{V'}$, $z^{-1}\theta_i^{-1} \in C_V$, and $ \phi=\frac{z^{-1}\theta_i^{-1}}{z^{-1}\theta_i^{-1}t+1} \in C_V(t) $. Thus $\phi \in \bigcap_{V \supseteq D'} C_V(t) = R$. 
In particular, since $\phi$ is a unit in $C_V$, we obtain $\phi^{-1} \in R^*_{\p}$ and therefore $z\theta_i \in R^*_{\p}$.
\end{proof}

In general the rings $A_V$ may be too small for being maximal excluding. In particular if $\overline{V}^{F}$ has more than two maximal ideals we do not expect $A_V$ to be maximal excluding. 
In some cases it can be shown that $A_V$ is maximal excluding.
For $D=K[[x^2, x^3, y]]$ and $D'=K[[x^2, y]]$, we have $A_V= V[x^3]_{\m_V, x^3}$. 
For $V=K[[x^2, y, \frac{y}{x^{2}}, \frac{y}{x^{4}}, \ldots]]$, the ring $A_V=K[[x^2, x^3, y, \frac{y}{x}, \frac{y}{x^{2}}, \ldots]]$ is maximal excluding since it can be expressed as pullback as in Theorem \ref{pullbacksquare}.
Also, if $V=K[[y, x^2, \frac{x^2}{y}, \frac{x^2}{y^2}, \ldots]]$, the ring $A_V=K[[y, x^3, x^2, \frac{x^2}{y}, \frac{x^2}{y^2}, \ldots]]$ is a generalized power series ring and is maximal with respect to excluding the element $\frac{x^3}{y}$ by Theorem \ref{maxexclseries}. 
In the next examples we construct rings of the form $Kr^{\mathcal{F}}(D)$ such that the rings $A_V$ or $C_V$ are all maximal excluding.

\begin{example}
\label{ex53}
Adopting the same notation of Example \ref{ex51} and Proposition \ref{propcostruction}, let 
$D'= \Q+ \mathfrak{n}$, $D= \Q+ \m= \Q+ \mathfrak{n} + i \mathfrak{n}$ and $\overline{D}= \Q(i)+\m$. For a valuation overring $V$ of $D'$, set $C_V = V[i\m_V]$. We show that $C_V$ is maximal with respect to excluding the element $i$. 
It is clear that $C_V \subseteq V(i)$ is a minimal ring extension.
We have two possible cases. In the first $a^2 + 1 \not \in \m_V$ for every $a \in V$. It is easy to check that in this case there is only one valuation $W$ of $F$ extending $V$, and $w(a+ib)= \min(v(a), v(b))$. Hence, $W=V(i)$ and its residue field is isomorphic to $\frac{V}{\m_V}(i)$. Now \cite[Theorem 2.4]{gil-hei}, or alternatively our Corollary \ref{k+mcase}, shows that $C_V$ is maximal excluding.
In the other case we have two distinct conjugated valuations $W_1$ and $W_2$ extending $V$ to $F$. We observe that $V(i)=W_1 \cap W_2$, indeed for $a+ib \in W_1 \cap W_2$ we cannot have $v(a)= v(b)<0$, otherwise the conjugate $a-ib$ would not be in the same ring. Again \cite[Theorem 2.4]{gil-hei} shows that $C_V$ is maximal excluding. The structure of $C_V$ in this case is described by \cite[Theorem 14]{jaballah}.
\end{example}

In the next example we replace the intersection of all the valuation overrings of $D'$ by the intersection of a smaller collection, which still form a defining family for $D'$. The Kronecker function ring associated to this family is an overring of the canonical Kronecker function ring.

\begin{example}
\label{ex54}
Let $\kappa$ be a field, $x$ an indetreminate over $\kappa$, and $T$ a local integrally closed domain of the form $\kappa(x)+\m$. Set $W=\kappa[[x]]$ and $B=\kappa[[x^2, x^3]]$ (one can also take any other maximal excluding domain having $W$ as integral closure). Define $D$ to be the pullback $ B + \m$. If we assume that $T$ is not a valuation ring, the ring $D$ is not maximal excluding by Theorem \ref{pullbacksquare}.
 For $F=\mathcal{Q}(D)$, let $F'$ be the largest subfield of $F$ containing $x^2$ but not $x$. The ring $D'$ in this case is set to be $\kappa[[x^2]]+ (\m \cap F')$ and we set also $T'= D'_{\m \cap F'}= \kappa(x^2)+(\m \cap F')$. The extension $F/F'$ satisfies the assumption of Construction \ref{construction}. We can now express $D'$ as the intersection of the valuation overrings $V$ of the form $V= \pi_V^{-1}(\kappa[[x^2]]) $ where $\pi_V$ is the quotient map from a valuation overring $V'$ of $T'$ to its residue field. 
 For any of such $V$, we can define as previously $A_V= V[\m_D]_{\m_V, \m_D} = V[x^3, \m]$. In particular $A_V$ is the pullback of $B$ with respect to the quotient map from the valuation overring $V'(x)$ of $T$ to its residue field. Therefore $A_V$ is maximal excluding by Theorem \ref{pullbacksquare}.
 Let $\mathcal{F}$ be the family of all the rings $A_V$. The same proof used for Theorem \ref{main1} shows that, if $R= Kr^{\mathcal{F}}(D)$, then $\overline{R}$ is the Kronecker function ring of $\overline{D}$ obtained by intersecting all the valuations of the form $V'(x)(t)$. Furthermore, all the localizations of $R$ at the maximal ideals are of the form $A_V(t)$.
\end{example}

We now move on to discuss a second construction which 
produces rings that are very close the the Kronecker function ring of $\overline{D}$, in the sense that they locally coincide for almost all the maximal ideals. This construction allows also to describe cases where the integral closure of $D$ is not finitely generated over $D$ or where the integral closure of some overring in the defining family of $D$ is not Pr\"{u}fer. The rings we define here are not necessarily of the form $Kr^{\mathcal{F}}(D)$, if the integral domain $A$ is neither maximal excluding nor has Pr\"{u}fer integral closure. However, Theorem \ref{VcapA} give us the tools to study a more general situation.

\begin{construction}
\label{construction2} \rm
Let $D$ be an integral domain and suppose that there exists a semilocal overring $A$ such that $D = \overline{D} \cap A$. Suppose that also $\overline{A}$ is semilocal with $s$ maximal ideals and all the residue fields of $D$ have cardinality at least $s+1$. 
Let $R'= Kr(\overline{D}) \cap A(t)$. 
\end{construction}

Notice that if $D$ arises as in Construction \ref{construction} and $A$ contains an intersection of rings of the form $A_V$, then the above ring $R'$ is an overring of the ring $R$, contained in $Kr(\overline{D})$. By Theorem \ref{main1}, one immediately gets $\overline{R'}= Kr(\overline{D})$. 
Easy examples that illustrate this construction are $$ D= K[[x^2, x^3, y]]= K[[x,y]] \cap K((y))[[x^2, x^3]],$$ and the ring $D$ defined in Example \ref{ex41}.

Before proving the main theorem about this construction, we need to investigate the integral closure of local rings of the form $A \cap V$ considered in Theorem \ref{VcapA}-(i). Obviously there is an inclusion $\overline{A \cap V} \subseteq \overline{A} \cap V$. However, the containment may be strict. For instance take $A= \Q + X\Q(i)[[X]]$ and $V= \Z_{(5)}[\frac{2+i}{5}] + X\Q(i)[[X]]$. These rings have the same quotient field and $V$ is a valuation domain.
We have $A \cap V = \Z_{(5)} + X\Q(i)[[X]]$, $\overline{A}= \Q(i)[[X]]$, and 
$$ \overline{A \cap V} = \Z_{(5)}(i) + X\Q(i)[[X]] \subsetneq \overline{A} \cap V = V. $$
Next proposition analyzes the relation between $\overline{A \cap V} $ and $ \overline{A} \cap V$.

\begin{prop}
\label{integralclosure}
Let $A$ be a local domain, $V$ a valuation ring such that $\mathcal{Q}(A)= \mathcal{Q}(V)= \mathcal{Q}(A \cap V)$ and suppose that $A$ is dominated by a localization $V_Q$ of $V$. Let $\q=Q \cap \overline{A} $ be the center of $V_Q$ in $\overline{A} $ and set $B= \overline{A \cap V}$ and $C = \overline{A}_{\q} \cap V$.
Then the following assertions hold:
\begin{enumerate}
\item[(1)] $ B':= \frac{B}{Q \cap B} $ and $C' := \frac{C}{Q \cap C}$ are Pr\"{u}fer domains with the same quotient field.
\item[(2)] $C'$ is the localization of $B'$ at the maximal ideal defined as the center of the valuation ring $ \frac{V}{Q} $.
\end{enumerate}
\end{prop}

\begin{proof}
Since $V_Q$ dominates $A$, then $\q$ is a maximal ideal of $\overline{A}$ and $V_Q$ dominates $\overline{A}_{\q}$. Thus $C$ is local and its structure is described in Theorem \ref{VcapA}-(i). In particular $Q \cap C = \q\overline{A}_{\q}$ and $C'$ is a valuation domain with quotient field $ \kappa= \frac{\overline{A}_{\q}}{\q\overline{A}_{\q} } = \frac{\overline{A}}{\q}.$ For the same reason, again by 
Theorem \ref{VcapA}-(i), we have that $A':= \frac{A \cap V}{\m_A} $ is a valuation domain with quotient field $\kappa'= \frac{A}{\m_A}$. By standard results on quotient rings of integral extensions, since $ Q \cap (A \cap V)= \m_A$, we get that $B'$ is an integral extension of $A'$ and $\kappa$ is an algebraic field extension of $\kappa'$. The integral closure of a valuation domain in an algebraic field extension of the quotient field is a Pr\"{u}fer domain. The ring $B'$ is clearly integrally closed, thus is a Pr\"{u}fer domain. It is clear that $B' \subseteq C'$. To prove (1) it remains to show that they have the same quotient field. For this we prove that $\overline{A} \cap V$ is contained in a ring of fractions $S^{-1}B$ for $S$ a multiplicatively closed set of $B$ with $S \cap Q = \emptyset$. Define $S= (A \cap V) \setminus \m_A$. This set $S$ is multiplicatively closed since $\m_A$ is a prime ideal of $A \cap V$ and $S \cap Q = S \cap \m_A= \emptyset$. Pick $x \in \overline{A} \cap V$. Then $x^n + a_{n-1}x^{n-1}+ \ldots + a_1x + a_0=0$ for $a_0, \ldots, a_{n-1} \in A$. If $a_0, \ldots, a_{n-1} \in V$ then $x \in B$. Thus suppose that at least one of them is not in $V$ and set $a= \prod_{a_k \not \in V} a_k$. Since $\m_A \subseteq Q \subseteq V$, we have that $a$ is a unit in $A$ and $a^{-1} \in V$. Multiplying the above equation of integral dependence of $x$ by $a^{-n}$ yields an equation of integral dependence of $a^{-1}x$ over $A \cap V$. It follows that $a^{-1}x \in B$ and $x \in S^{-1}B$ because $a^{-1} \in S$. From this, going modulo the contraction of $Q$ and observing that $ \frac{\overline{A} \cap V}{\q} $ and $C'$ have the same quotient field, 
we obtain that $B'$ and $C'$ have the same quotient field. Finally to prove (2), we simply notice that $V$ dominates both $A \cap V$ and $C$, therefore it dominates the localization of $B$ at the maximal ideal $ \m_V \cap B = \m_C \cap B$. It follows that $C'$ dominates the localization of $B'$ at the center of $\frac{V}{Q}$. Since $B'$ and $C'$ are Pr\"{u}fer domains with the same quotient field, $C'$ must be equal to that localization of $B'$.
\end{proof}

\begin{theorem}
\label{main2}
Let $D$, $A$ and $R'$ be defined as in Construction \ref{construction2}. The following assertions hold:
\begin{enumerate}
\item[(i)] For a valuation overring $V$ of $D$, let $\p= \m_{V(t)} \cap R'$. Then $R'_{\p}= (A \cap V)_{\m_{V} \cap A}(t)$. Moreover, if $A$ and $V$ have no common proper overrings 
then $R'_{\p} = V(t)$. 
\item[(ii)] For every valuation overring $\mathcal{V}$ of $R'$, there exists a valuation overring $V$ of $D$ such that $\mathcal{V} \supseteq (A \cap V)(t)$.
\item[(iii)] 
The localizations of $R'$ at its maximal ideals are Nagata ring extensions of overrings of $R$.
\item[(iv)] 
The integral closure of $R'$ is
$Kr(\overline{D}) \cap \overline{A}(t)$. 
\end{enumerate}
\end{theorem}

\begin{proof}
It is well-known that the number of maximal ideals of $A$ is less or equal than the number of maximal ideals of $\overline{A}$. 
The assumption on the residue fields of $D$ implies that also the residue fields of $A \cap V$ and $\overline{A} \cap V$ have cardinality at least $s$ for every valuation overring $V$ of $D$. Therefore the intersections $A \cap V$ and $\overline{A} \cap V$ satisfy the hypothesis of Theorem \ref{VcapA}. \\
(i) 
By Theorem \ref{VcapA}-(iv) 
we immediately get $(A \cap V)(t)= A(t) \cap V(t)$. 
Hence, $R' \subseteq (A \cap V)(t)$ and $\p$ is equal also to the contraction in $R'$ of the maximal ideal of $(A \cap V)_{\m_{V} \cap A}$. This implies $R'_{\p} \subseteq (A \cap V)_{\m_{V} \cap A}(t)$. We only need to show $(A \cap V)_{\m_{V} \cap A} \subseteq R'$ and apply Lemma \ref{localizationlemma}. 
Let $a,s \in A \cap V$ with $s \not \in \mathfrak{m}_V$. Consider 
$$F = \frac{1}{t+ a},  \quad G = \frac{1}{t+ s^{-1}} = \frac{s}{st+1}. $$ 
Clearly $F,G \in Kr(\overline{D}) \cap A(t) = R'$. Moreover, $v(a) \geq 0$, $v(s)=0$, and $F, G$ are units in $V(t)$. Thus $F^{-1}= t+a$ and $G^{-1}= t+s^{-1}$ are in $  R'_{\mathfrak{p}}$. This proves $\frac{a}{s} \in R'_{\mathfrak{p}},$ because $t \in R'$.
If $A$ and $V$ have no common proper overrings we get $R'_{\p} = V(t)$ by Theorem \ref{VcapA}-(iii). \\
(ii) Set
$V= \mathcal{V} \cap \mathcal{Q}(D)$ (of course we allow also the case $V=\mathcal{Q}(D) $). Clearly $V \supseteq R' \cap \mathcal{Q}(D) \supseteq D(t) \cap \mathcal{Q}(D) = D$.
Since $t \in R' $ we get $(A \cap V)[t] \subseteq V[t] \subseteq \mathcal{V}$. Pick a polynomial $f=\sum_{k=0}^m a_k t^k \in (A \cap V)[t]$ such that $c(f)= A \cap V$. Let $a_k$ be a coefficient of $f$. In every valuation overring of $Kr(\overline{D})$, the value of $a_k$ is larger than or equal to the value of $f$.
Hence $ a_kf^{-1} \in Kr(\overline{D}) \cap A(t) =  R' \subseteq \mathcal{V}. $ Since $c(f)= A \cap V$ and $A$ is semilocal, by Lemma \ref{semilocalunits} there exist $u_0, \ldots, u_m \in D$, that are either units or zeros, such that $a= \sum_{k=0}^m a_k u_k $ is a unit in $ A \cap V $. It follows that $af^{-1} \in R' \subseteq \mathcal{V}$. Now $a$ is a unit in $A \cap V$, hence a unit in $\mathcal{V}$. This implies $f^{-1} \in \mathcal{V}$ and shows $(A \cap V)(t) \subseteq \mathcal{V}$. \\
(iii) Every prime ideal of $R'$ is the center of some valuation overring of $R'$. 
By item (ii), every valuation overring $\mathcal{V}$ of $R'$ contains an intersection of the form $ (A \cap V)(t) = A(t) \cap V(t)$ for some valuation overring $V$ of $D$. 
By item (i), $(A \cap V)_{\m_{V} \cap A}(t)$ is the localization of $R'$ at the center of the valuation overring $V(t)$. In particular, given a maximal ideal $\m$ of $A$, the ring $A_{\m}(t)$ is the localization of $R'$ at the center of any valuation overring $V(t)$ such that $V$ dominates $A_{\m}$. 
Let $\p= \m_{\mathcal{V}} \cap R'$. By Theorem \ref{VcapA}, since localization at the maximal ideals commutes with Nagata ring extension, $\mathcal{V}$ contains a ring $C(t)$ which is equal either to $A_{\m}(t)$ for $\m$ a maximal ideal of $A$ or to $(A \cap V)_{\m_{V} \cap A}(t)$. For what said above, this ring $C(t)$ is a localization of $R'$ at a prime ideal and, if $\q= \m_{\mathcal{V}} \cap C(t)$, then $\q \cap R' = \p$.
 Hence $R'_{\p}$ is the localization of $C(t)$ at $\q$. In particular, since we already know that the rings $C(t)$ are localizations of $R'$ at some prime ideal, it follows that the localizations of $R'$ at its maximal ideals are all of the form $C(t)$ for some overring $C$ of $D$. \\
(iv) We argue showing that every valuation overring of $R'$ contains $\overline{A}(t) \cap V(t) = (\overline{A} \cap V)(t)$ 
for some valuation overring $V$ of $D$. Let $\mathcal{V}$ be a valuation overring of $R'$ and let $V= \mathcal{V} \cap \mathcal{Q}(D)$. By item (ii) we already know that $ \mathcal{V} \supseteq (A \cap V)(t) $. If $\mathcal{V} \supseteq A(t)$, we are done since then $\mathcal{V} \supseteq \overline{A(t)} = \overline{A}(t)$.
Otherwise by Theorem \ref{VcapA}, localizing at the maximal ideals of $A \cap V$ we find that $ \mathcal{V} \supseteq (A \cap V)_{\m_V \cap A}(t). $ If $(A \cap V)_{\m_V \cap A} = V $ we can conclude. If not, again by Theorem \ref{VcapA}, we reduce to the case where there exists a prime ideal $Q$ of $V$ such that $V_Q$ contains $A$ and $(A \cap V)_{\m_V \cap A} = A_{Q \cap A} \cap V.  $ Without loss of generality, we simplify the notation possibly replacing $A$ by $A_{Q \cap A}$, and reduce to the case where $A$ and $A \cap V$ are local, $V_Q$ dominates $A$, and $ \mathcal{V} \supseteq (A \cap V)(t) $. Obviously, using that Nagata ring extension commutes with integral closure, we get $ \mathcal{V} \supseteq (\overline{A \cap V})(t) $.
 Now, if $\overline{A \cap V}= \overline{A} \cap V$ we are clearly done. If not we use the result of Proposition \ref{integralclosure}. Let $\q=Q \cap \overline{A} $ and set $B= \overline{A \cap V}$ and $C = \overline{A}_{\q} \cap V$.
  It is sufficient to prove that $\mathcal{V} \supseteq C(t)$. Since $\mathcal{V}$ contains $V$ and $t$, then $\mathcal{V} \supseteq C[t]$. Notice also that the maximal ideal $\m_V \cap B$ of $B$ is contained in $\m_{\mathcal{V}}$. Thus, by localization we find that $\mathcal{V} \supseteq B_{\m_V \cap B}(t)$.
 Pick $f \in C[t]$ such that $c(f)= C$. We want to prove that $f^{-1} \in \mathcal{V}$. For $h \in C[t]$, if $h^{-1} \in \mathcal{V}$ and $g \in \m_V[t] \subseteq \m_{\mathcal{V}}$ we get $(h+g)^{-1} \in \mathcal{V}$. Hence, we want to find $g \in \m_V[t]$ such that $f+g \in B_{\m_V \cap B}[t]$ and $c(f+g)=B_{\m_V \cap B}$. But Proposition \ref{integralclosure} implies that $B_{\m_V \cap B}$ and $C$ coincide after going modulo the contractions of $Q$ on both rings. In particular, if $x \in C$, then there exists $z \in Q \cap C \subseteq Q \subseteq \m_V$ such that $x+z \in B_{\m_V \cap B}$. Furthermore, if $x$ is a unit in $C$ then $x+z$ is a unit in $B_{\m_V \cap B}$. Applying this procedure to all the coefficients of $f$ that are not in  $B_{\m_V \cap B}$, one can construct a polynomial $g \in Q[t] \subseteq \m_V[t]$ such that $f+g \in B_{\m_V \cap B}[t]$ and $c(f+g)=B_{\m_V \cap B}$.
\end{proof}


\begin{remark}
With the setting of the above theorem, if $\overline{A}$ is Pr\"{u}fer, we automatically have $\overline{R'}= Kr(\overline{D})$.
In this case, since we assume all the rings to have finite dimension,
$\overline{A}$ has only finitely many valuation rings, hence all but finitely many localizations of $R'$ at prime ideals are valuation domains of the form $V(t)$. The other localizations are localizations of $A(t)$ or of domains of the form $(A \cap V)(t)$.

 For instance, let $D= K[[x^2, x^3, y]]$, $ A= K((y))[[x^2, x^3]]$, and $R'= Kr(\overline{D}) \cap A(t)$. 
Then $\overline{R'}= Kr(\overline{D})$ because $\overline{A}$ is Pr\"{u}fer. 
In this case all but two localizations of $R'$ at its prime ideals are of the form $V(t)$. The two remaining localizations are $A(t)$ and $(A \cap W)(t)$ where $W$ is the rank two valuation overring of $D$ having $\overline{A}$ as unique minimal overring. 
\end{remark}




We construct now examples of Kronecker function ring of $D$ of the form $Kr(\overline{D}) \cap A(t)$ such that $A$ is maximal excluding, but $\overline{A}$ is not a valuation. Therefore $\overline{R'} = Kr(\overline{D}) \cap \overline{A}(t) \subsetneq Kr(\overline{D})$.

\begin{example}
\label{lastexampls}
Let $K$ be any field and let $D'= [[K^S]]$ be a maximal excluding generalized power series ring. 
Let $z$ be an indeterminate over $\mathcal{Q}(D)$ and set $D = D'[[z]]$. Let $A$ be the generalized power series ring over $K$ obtained by adding to $D$ all the elements of the form $\frac{z}{y}$ for $y \in \m_{D'}$ a monomial element. This ring $A$ can be represented as pullback of $D'$ with respect to the quotient map from $\mathcal{Q}(D')[[z]]$ to its residue field. Hence, by Theorem \ref{pullbacksquare}, $A$ is maximal excluding. By standard properties of pullback diagrams (see \cite{Gabelli-Houston}), $ \overline{A} $ is the pullback of $\overline{D'}$ with respect to the same quotient map. 
Moreover, $D = \overline{D} \cap A$. Let $R'= Kr(\overline{D}) \cap A(t)$. 

Notice that $A$ and $\overline{A}$ are both local, therefore $D$, $A$ and $R'$ satisfy the hypothesis of Construction \ref{construction2} without restrictions on the field $K$. 
 The generalized power series rings of Example \ref{ex3} and Example \ref{ex4} provide examples where $\overline{A}$ is not a Pr\"{u}fer domain. Those from Examples \ref{ex1} and \ref{ex2} provide examples where $\overline{D}$ is not finitely generated as $D$-module.  
\end{example}

\section*{Acknowledgements}
The first author is supported by the grants MAESTRO NCN -
UMO-2019/34/A/ST1/00263 - Research in Commutative Algebra and
Representation Theory and NAWA POWROTY - PPN/PPO/2018/1/00013/U/00001 - Applications of Lie algebras to Commutative Algebra.
The first author would like to thank Ohio State University for supporting his research visit in Columbus in May 2022.  
Both authors would like to thank Marco D'Anna and Carmelo A. Finocchiaro for giving them the possibility to visit University of Catania in June 2022 and for interesting conversations about the content of this paper.

\end{document}